\documentclass[10pt,twoside, a4paper, english, reqno]{amsart}

\usepackage{hyperref}

\DeclareMathSizes{7}{18}{12}{8}
\theoremstyle{plain}
\newtheorem{thm}{Theorem}[section]
\theoremstyle{plain}
\newtheorem{lem}[thm]{Lemma}
\newtheorem{prop}[thm]{Proposition}

\theoremstyle{definition}

\newtheorem{rem}{Remark}[section]

\newenvironment{Assumptions}
{%
\setcounter{enumi}{0}

\begin{enumerate}}%
{\end{enumerate} }
{%
\setcounter{enumi}{0}

\begin{enumerate}}%
{\end{enumerate} }

\newcommand{\alp}{\alpha}

\newcommand{\eps}{\ensuremath{\varepsilon}}
\newcommand{\R}{\ensuremath{\mathbb{R}}}

\newcommand{\Dx}{\ensuremath{\Delta x}}
\newcommand{\Dz}{\ensuremath{\Delta z}}
\newcommand{\Dt}{\ensuremath{\Delta t}}

\DeclareMathOperator*{\argmin}{argmin}

\def\N{\mathbb{N}}
\def\Z{\mathbb{Z}}

\numberwithin{equation}{section} \allowdisplaybreaks

\title[Difference-quadrature schemes for IPDEs]
{Difference-quadrature schemes for nonlinear degenerate parabolic integro-PDE}

\date{\today}

\author[I. H. Biswas]{I. H. Biswas}
\address[Imran H. Biswas]{\newline
Seminar for Applied Mathematics,
ETH,
CH-8092 Zurich,
Switzerland}
\email[]{imran.biswas@sam.math.ethz.ch}

\author[E. R. Jakobsen]{E. R. Jakobsen}
\address[Espen R. Jakobsen]{\newline
Norwegian University of Science and Technology, NO--7491, Trondheim, Norway} \email[]{erj@math.ntnu.no}

\author[K. H. Karlsen]{K. H. Karlsen}
\address[Kenneth Hvistendahl Karlsen]{\newline
Centre of Mathematics for Applications,
University of Oslo,
P.O.\ Box 1053, Blindern,
NO--0316 Oslo, Norway}
\email[]{kennethk@math.uio.no}
\urladdr{folk.uio.no/kennethk}



\thanks{This work was supported by the Research Council of Norway
  (NFR) through the project "Integro-PDEs: Numerical methods,
  Analysis, and Applications to Finance". The work of K. H. Karlsen
  was also supported trough a NFR Outstanding young Investigator Award. 
This article was written as part of the the international research program
on Nonlinear Partial Differential Equations at the Centre for
Advanced Study at the Norwegian Academy of Science
and Letters in Oslo during the academic year 2008--09.}

\begin{document}

\begin{abstract}
We derive and analyze monotone difference-quadrature schemes for
Bellman equations of controlled L\'{e}vy (jump-diffusion)
processes. These equations are fully non-linear, degenerate parabolic
integro-PDEs interpreted in the sense of viscosity solutions. We
propose new ``direct'' discretizations of the non-local part of the
equation that give rise to monotone schemes capable of handling
singular L\'{e}vy measures.  Furthermore, we develop a new general
theory for deriving error estimates for approximate solutions of
integro-PDEs, which thereafter is applied to the proposed
difference-quadrature schemes.
\end{abstract}

\maketitle

\tableofcontents

\section{Introduction}\label{sec:intro}
In this article we derive and analyze numerical schemes for fully
non-linear, degenerate parabolic integro partial differential
equations (IPDEs) of Bellman type. To be precise, we
consider the initial value problem
\begin{align}\label{eq:integro_pde} 
	u_t+\sup_{\alpha\in\mathcal{A}}\Big\{-L^\alp[u](t,x)+c^\alpha(t,x)u
	-f^\alpha(t,x)-J^{\alpha}[u](t,x)\Big\}= 0&
	\quad \textrm{in}\ Q_T,\\
	\label{eq:bdcond}u(0,x)= g(x)& \quad \text{in}\
\mathbb{R}^N,
\end{align}
where $Q_T:= (0,T]\times\mathbb{R}^N$ and
\begin{align*}
	&L^\alp[\phi](t,x):=\textrm{tr}\big[a^\alpha(t,x)D^2\phi\big]
	+b^\alpha(t,x) D\phi,\\
	&J^{\alpha}[\phi](t,x) :=\int_{\mathbb{R}^M\backslash \{0\}}
	\Big(\phi(t,x+\eta^\alpha(t,x,z))-\phi-{\bf 1}_{|z|\le 1}
	\eta^{\alpha}(t,x,z)D  \phi\Big)\nu(dz),
\end{align*}
for smooth bounded functions $\phi$. Equation \eqref{eq:integro_pde}
is convex and non-local. The coefficients $a^\alpha,\eta^\alpha,
b^\alpha,c^\alpha, f^\alpha,g$ are given functions taking 
values respectively in $\mathbb{S}^N$ ($N\times N$
symmetric matrices), $\mathbb{R}^N$, $\mathbb{R}^N$, $\R$, $\R$, and 
$\R$. The L\'{e}vy measure $\nu(dz)$ is a positive, possibly singular, Radon
measure on $\mathbb{R}^M\backslash\{0\}$; precise assumptions will be given later.

The non-local operators $J^\alp$ can be pseudo-differential
operators. Specifying $\eta\equiv z$ and $\nu(dz)=\frac K{|z|^{N+\gamma}}\,dz$, 
$\gamma\in(0,2)$, give rise to the fractional 
Laplace operator $J=(-\Delta)^{\gamma/2}$. These operators are allowed to
degenerate since we allow $\eta=0$ for $z\neq 0$. The second order
differential operator $L^\alp$ is also allowed to
degenerate since we only assume that the diffusion matrix
$a^{\alpha}$ is nonnegative definite. Due to these two
types of degeneracies, equation \eqref{eq:integro_pde} 
is {\em degenerate parabolic} and there is no (global) smoothing of
solutions in this problem (neither ``Laplacian'' nor ``fractional
Laplacian'' smoothing). Therefore equation \eqref{eq:integro_pde}
will have no classical solutions in general. From the type
non-linearity and degeneracy present in \eqref{eq:integro_pde} the
natural type of weak solutions are the viscosity solutions \cite{Crandall:1992ie,Fleming:1993dy}. 
For a precise definition of viscosity solution of \eqref{eq:integro_pde} we refer to
\cite{Jakobsen:2005jy}. In this paper we will work with H\"older/Lipschitz
continuous viscosity solution of \eqref{eq:integro_pde}-\eqref{eq:bdcond}. 
For other works on viscosity solutions and IPDEs of second order, we refer to
\cite{Alvarez:1996lq, Amadori:2003bo,
   Barles:1997xj, BI:P07, BCI:P07, Benth:2001tv,
   CS:P07,Jakobsen:2005jy,Jakobsen:2006aa,Mikulyavichyus:1996oi,
   Sayah:1991sk} and references therein.

Nonlocal equations such as \eqref{eq:integro_pde} appear as the
dynamic programming equation associated with optimal control of
jump-diffusion processes over a finite time horizon (see
\cite{Mikulyavichyus:1996oi, Pham:1998zt, Biswas:2007td}). Examples
of such control problems include various portfolio optimization
problems in mathematical finance where the risky assets are driven
by L\'{e}vy processes. The linear pricing equations for European and Asian
options in L\'{e}vy markets are also of the form \eqref{eq:integro_pde}
if we take $\mathcal{A}$ to be a singleton.
For more information on pricing theory and its relation to
IPDEs we refer to \cite{Cont:2004gk}.

For most nonlinear problems like \eqref{eq:integro_pde}-\eqref{eq:bdcond}, 
solutions must be computed by a numerical scheme. 
The construction and analysis of numerical schemes for nonlinear 
IPDEs is a relatively new area of research. 
Compared to the PDE case, there are currently only a few works available. 
Moreover, it is difficult to prove that such schemes converge 
to the correct (viscosity) solution. In the literature there are two main strategies for 
the discretization the non-local term in \eqref{eq:integro_pde}. 
One is indirect in the sense that the L\'{e}vy measure is first 
truncated to obtain a finite measure and then the
corresponding finite integral term is approximated by a quadrature rule. 
Regarding this strategy, we refer to \cite{Cont:2004gk,CV} (linear 
or obstacle problems) and \cite{Jakobsen:2005tp,CJ:08} (general non-linear problems). 
The other approach is to discretize the integral term directly. Now there are 3
different cases to consider depending on whether (i) $\int_{|z|<1} \nu(dz)<\infty$,
(ii) $\int_{|z|<1} |z|\nu(dz)<\infty$, or (iii) $\int_{|z|<1} |z|^2
\nu(dz)<\infty$. Case (i) is the simplest one and has been 
considered by many authors, see, e.g., 
\cite{Kushner:1992mq,Cont:2004gk,BLN04,DFV05,AO07_2,Jakobsen:2005tp} and 
references therein. Case (ii) was considered in 
\cite{AO07,Schwab1,FRS08}, and case (iii) in \cite{Schwab1,FRS08}. 
Most of the cited papers restrict their attention to linear, non-degenerate, 
one-dimensional equations or obstacle problems for such equations.

One of the contributions of this paper is a class of 
direct approximations of the non-local part of \eqref{eq:integro_pde}, 
giving rise to new monotone schemes that are capable of handling 
singular L\'{e}vy measures and moreover are supported by a theoretical analysis. 
The proposed schemes are new also in the linear case. 
As in \cite{AO07} (cf.~also \cite{Lynch:2003yq} for 
a related approach), the underlying idea is to perform integration by parts 
to obtain a bounded ``L\'{e}vy''  measure and an integrand 
involving derivatives of the unknown solution. 
In \cite{AO07}, one-dimensional, constant coefficients, 
linear equations (and obstacle problems) are discretized 
under the assumption $\int_{|z|<1} |z|\, \nu(dz)<\infty$.  
Their schemes are high-order and non-monotone, but not supported by 
rigorous stability and convergence results. 
In this paper we discretize general non-linear, multi-dimensional, 
non-local equations without any additional restrictive 
integrability condition on the L\'{e}vy measure. 
More precisely, we provide monotone difference-quadrature 
schemes for \eqref{eq:integro_pde}-\eqref{eq:bdcond} and prove under weak assumptions 
that these schemes converge with a rate to the exact viscosity solution of 
the underlying IPDE. The schemes we put forward and our convergence
results apply in much more general situations than 
those previously treated in the literature.

The second main contribution of this paper is a theory of error 
estimates for a class of monotone approximations schemes 
for the initial value problem \eqref{eq:integro_pde}-\eqref{eq:bdcond}. 
We use this theory to derive error estimates for the proposed numerical schemes. 
For IPDEs in general and non-linear IPDEs in particular, there 
are few error estimates available, see \cite{CV,Schwab2} for linear equations and \cite{Jakobsen:2005tp,Biswas:2006tr,CJ:08} for non-linear equations.

Error estimates involving viscosity solutions first appeared in 
1984 for first order PDEs \cite{CL:2approx}, in 1997/2000
for convex 2nd order PDEs \cite{Krylov:1997nf, Krylov:2000yy},
and in 2005/2008 for IPDEs \cite{CV,Jakobsen:2005tp}.
The results obtained for IPDEs, including those in this paper, are
extensions of the results known for convex second order PDEs, which are 
based on Krylov's method of shaking the coefficients \cite{Krylov:2000yy}. 
Krylov's method produces smooth approximate subsolutions of the equation (or scheme) 
that, via classical comparison and consistency arguments, imply one-sided error estimates.
Based on this idea, there are currently two types of error 
estimates for convex second order PDEs: 
(i) optimal rates applying to specific schemes and equations
(cf., e.g., \cite{Krylov:2005lj,Jakobsen:2003nu}) and 
(ii)  sub-optimal rates that apply to ``any'' monotone consistent
approximation (cf., e.g., \cite{Barles:2006jf,Krylov:2000yy}). In
particular, type (i) results apply when you have a priori
regularity results for the scheme, while type (ii) results do not require this.

In this paper we provide error estimates of type (ii), whereas 
earlier results for IPDEs are of type (i), see 
\cite{CV,Jakobsen:2005tp,Biswas:2006tr,CJ:08}. The problem with 
type (i) results is the difficulty in establishing the required 
priori regularity estimates. In the PDE case this can be achieved for
particular schemes \cite{Krylov:2005lj,Jakobsen:2003nu}, and 
attempts to generalize these schemes to the IPDE setting have only been partially 
successful \cite{Biswas:2006tr,CJ:08}, since the required 
regularity estimates have been obtained only through unnaturally 
strong restrictions on the non-local terms. 
In \cite{CJ:08} the L\'{e}vy measure is bounded 
and in \cite{Biswas:2006tr} the L\'{e}vy measure 
is either bounded or the integral term is independent of $x$ with an 
(essentially) one-dimensional L\'{e}vy measure.  
Of course, by a truncation procedure only bounded L\'{e}vy measures need to be 
considered \cite{Cont:2004gk,Jakobsen:2005tp}, but such
approximations may not be accurate and the resulting
error estimates blow up as the truncation parameter tends to zero.
An advantage of the error estimates in the present paper is that they
apply without any such restrictions. 
In particular, we can handle naturally any 
singular L\'{e}vy measures directly in our framework.

To prove our results we extend the approach of \cite{Barles:2006jf} 
to the non-local setting. To this end, we have to invoke 
a switching system approximation of \eqref{eq:integro_pde} (see Section \ref{sec:sw}). 
Switching systems of this generality have not been studied before. 
In paper \cite{Biswas:2007td}, we provide well-posedness, 
regularity, and continuous dependence results 
for such systems. We also prove that the value function of a
combined switching and continuous control problem solve the switching 
system under consideration.

The remaining part of this paper is organized as follows: 
First of all, we shall end this introduction by listing some relevant notation. 
In Section \ref{sec:well-posed} we list a few standing assumptions 
and provide corresponding well-posedness and regularity 
results for the IPDE problem \eqref{eq:integro_pde}-\eqref{eq:bdcond}.  
In Section \ref{sec:FD} we present a rather general approximation scheme 
for this problem, and show that it is consistent, monotone, and convergent. 
Error estimates for general monotone approximation schemes are stated in Section \ref{sec:ErrBnd}. 
In Section \ref{sec:ApprInt} we present new direct discretizations of the non-local term in
\eqref{eq:integro_pde}, and prove that these discretizations are consistent, 
monotone, and also satisfy the requirements introduced in Section \ref{sec:FD}.  
The switching system approximation of \eqref{eq:integro_pde} 
is introduced and analyzed in Section \ref{sec:sw}. 
The obtained results are utilized in Section \ref{sec:PfErr} 
to prove the error estimate stated Section \ref{sec:ErrBnd}. 
Finally, in Appendix \ref{sec:Lh} we give a standard example of
a (monotone) discretization of the local PDE part of \eqref{eq:integro_pde}
that satisfies the requirements of Section \ref{sec:FD}.

We now introduce the notation we will use in this paper. By
$C, K$ we mean various constants which may change from line to line.
The Euclidean norm on any $\mathbb{R}^d$-type space is
denoted by $|\cdot|$. For any subset $Q\subset \mathbb{R}\times
\mathbb{R}^N$ and for any bounded, possibly vector valued, 
function on $Q$, we define the following norms,
\begin{align*}
	|w|_0 := \sup_{(t,x)\in Q} |w(t,x)|,\qquad 
	|w|_{1} = |w|_0 + \sup_{(t,x)\neq(s,y)}
	\frac{|w(t,x)-w(t,y)|}{|t-s|^{\frac 12}+|x-y|}.
\end{align*}
Note that if $w$ is independent of $t$, then $|w|_1$ is the
Lipschitz (or $W^{1,\infty}$) norm of $w$. We use $C_b(Q)$ to denote the
space of bounded continuous real valued functions on $Q$.
Let $\rho(t,x)$ be a smooth and non-negative function on
$\mathbb{R}\times\mathbb{R}^N$ with unit mass and support in $\{0 < t< 1\}\times
\{|x|<1\}$. For any $\epsilon > 0$, we define the 
mollifier $\rho_{\epsilon}$ by
\begin{equation}\label{eq:mollifier} 
	\rho_{\epsilon}(t,x) := 
	\frac{1}{\epsilon^{N+2}}\rho\Big(\frac{t}{\epsilon^2},
	\frac{x}{\epsilon}\Big).
\end{equation}
In this paper we denote by $h$ the vector
$$
h=(\Dt,\Dx,\Dz)>0,
$$
and any dependence on $\Dt$, $\Dx$, or $\Dz$ will be denoted by
subscript $h$. The grid is denoted by $\mathcal G_h$ and is a subset
of $\bar Q_T$ which need not be uniform or even discrete in
general. We also set $\mathcal G_h^0=\mathcal G_h\cap\{t=0\}$ 
and $\mathcal G_h^+=\mathcal G_h\cap\{t>0\}$.

\section{Well-posedness \& regularity results for the Bellman equation}\label{sec:well-posed}

In this section we give some relevant well-posedness and regularity results for the Bellman 
equation \eqref{eq:integro_pde}-\eqref{eq:bdcond}.  
To this end, we impose the following assumptions:
\begin{Assumptions}
\item\label{A1} The control set $\mathcal{A}$ is
a separable metric space. For any $\alpha \in\mathcal{A}$,
$a^\alpha =
\frac{1}{2}\sigma^\alpha{\sigma^{\alpha}}^T$, and
$\sigma^\alpha, b^\alpha$, $c^\alpha, f^\alpha, \eta^\alpha$
are continuous in $\alpha$ for all $x,t, z$.
\medskip

\item\label{A2} There is
a positive constant $K$ such that for all
$\alp\in\mathcal A$,
\begin{align*}
|g|_1+|\sigma^\alpha|_1+|b^\alpha|_1+|c^\alpha|_1+|f^\alpha|_1\le K.
\end{align*}

\item\label{A3} For every $\alpha\in \mathcal{A}$ and
$z\in \mathbb{R}^M$ there is an $\Lambda\geq0$ such that
\begin{align*}
|e^{-\Lambda
|z|} \eta^\alpha(\cdot,\cdot,z)|_1\le K(|z|\wedge
1) \quad\text{and}\quad
|e^{- \Lambda |\cdot|}\eta^\alpha(t,x,\cdot)|_1\le K.
\end{align*}

\item\label{A4}
$\nu$ is a positive Radon measure on
$\R^M\setminus\{0\}$ satisfying
\begin{align*}
\int_{0<|z|\le 1} |z|^2 \nu(dz)+\quad \int_{|z|\ge 1}
e^{(\Lambda+\epsilon)|z|}\nu(dz)\le K
\end{align*} for some $K\geq0$, $\epsilon > 0$ where $\Lambda$ is
defined in \ref{A3}.
\end{Assumptions}

Sometimes we need the following 
stronger assumptions than \ref{A3} and \ref{A4}:
\begin{Assumptions}
\item[({\bf A}.4')] $\nu$ is a positive Radon 
measure having a density $k(z)$ satisfying
$$
0\leq k(z)\leq\frac{e^{-(\Lambda+\eps)|z|}}{|z|^{M+\gamma}}\quad
\text{for all}\quad z\in\R^M\setminus\{0\},
$$
for some $\gamma\in(0,2)$, $\eps>0$, 
where $\Lambda$ is defined in \ref{A3}.
\medskip

\item[({\bf A}.5)] Assume that \ref{A3} holds and let $\gamma$ as in
({\bf A}.4'). There is a constant $K$ such that for
every $\alpha\in \mathcal{A}$ and $z\in \mathbb{R}^M$
$$
|D_z^k\eta^\alp(\cdot,\cdot,z)|_0
+|D_x^l\eta^\alp(\cdot,\cdot,z)|_0 
\leq Ke^{\Lambda|z|},
$$
for all
\begin{align*}
	\begin{array}{ll}
		k=l=1 & \text{ when $\gamma=0$},\\
		k,l\in\{1,2\}& \text{ when $\gamma\in(0,1)$},\\
		k\in\{1,2,3,4\},\ l\in\{1,2\}& 
		\text{ when $\gamma\in[1,2)$.}
	\end{array}
\end{align*}
\end{Assumptions}

Assumptions \ref{A1}--\ref{A4} are standard and general.
The assumptions on the non-local term are motivated by applications in
finance. Almost all L\'{e}vy models in finance are covered by these
assumptions. It is easy to modify the results in this paper so that
they apply to IPDEs under different assumptions on the L\'{e}vy
measures, e.g., to IPDEs of fractional Laplace type where there is no
exponential decay of the L\'{e}vy measure at infinity.
Finally, assumption (A.5) is not strictly speaking needed in this
paper. We use it in some results because it 
simplifies some of our error estimates.

Under these assumptions the following results hold:

\begin{prop}\label{wellpos_1}
Assume \ref{A1}--\ref{A4}.
\smallskip

\noindent (a) There exists a unique bounded  viscosity solution $u$ of
the initial value problem \eqref{eq:integro_pde}--\eqref{eq:bdcond}
satisfying $|u|_1<\infty$.
\smallskip

\noindent (b) If $u_1$ and $u_2$ are respectively viscosity sub and supersolutions of
\eqref{eq:integro_pde} satisfying $u_1(0,\cdot)\le u_2(0,\cdot)$, then $u_1\le u_2$.
\end{prop}

The precise definition of viscosity solutions for the non-local
problem \eqref{eq:integro_pde}--\eqref{eq:bdcond} and the proof of
Proposition \ref{wellpos_1} can be found in \cite{Jakobsen:2005jy},
for example.
\section{Difference-Quadrature schemes for the Bellman equation}\label{sec:FD}
Now we explain how to discretize \eqref{eq:integro_pde}--\eqref{eq:bdcond} 
by convergent monotone schemes on a uniform grid (for simplicity). We
start by the spatial part and approximate the non-local part $J^\alp$
as explained later in Section \ref{sec:ApprInt} and the local PDE part
$L^\alp$ by a standard monotone scheme (cf.~\cite{Kushner:1992mq} and
Appendix \ref{sec:Lh}).  
The result is a system of ODEs in $\Dx\, \Z^N\times (0,T)$:
$$
u_t +\sup_{\alpha\in\mathcal{A}}\Big\{-L_h^\alp[u](t,x)
+c^\alpha(t,x)u-f^\alpha(t,x)-J_h^\alpha[u](t,x)\Big\}=0,
$$
where $L_h$ and $J_h$ are monotone, consistent 
approximations of $L$ and $J$, respectively. 

Then we discretize in the time variable 
using two separate $\theta$-methods, one for
the differential part and one for the integral part. For
$\vartheta,\theta\in[0,1]$, the fully discrete scheme reads
\begin{align}\label{FD}
&U^{n}_\beta
=U_\beta^{n-1}-\Dt \sup_{\alpha\in\mathcal{A}}
\Big\{-\theta L_h^\alp[U]^{n}_\beta-(1-\theta)
L_h^\alp[U]^{n-1}_\beta+c^{\alpha,n-1}_\beta U^{n-1}_\beta
\\ &\qquad\qquad\qquad-f^{\alpha,n-1}_\beta 
-\vartheta J_h^\alpha[U]^{n}_\beta-(1-\vartheta)J_h^\alpha[U]^{n-1}_\beta\Big\}
\qquad \text{in}\qquad
\mathcal G_h^+,\nonumber\\
& U_\beta^0=g(x_\beta)\qquad \text{in}\qquad 
\mathcal G_h^0,\label{FD_BC}
\end{align}
where $\mathcal G_h=\Dx \Z^N\times \Dt \{0,1,2,\dots,\frac T \Dt\}$ 
and $U^n_\beta=U(t_n,x_\beta)$, $f^{\alp,n}_\beta
=f^\alp(t_n,x_\beta)$, etc., for $t_n=n\Dt$ ($n\in\N_0$) and
$x_\beta=\beta \Dx$ ($\beta\in\Z^N$).

The approximations $L_h$ and $J_h$ are consistent, satisfying
\begin{align}
& |L^\alp[\phi]-L^\alp_h[\phi]|\leq K_L (|D^2\phi|_0\Dx
+|D^4\phi|_0\Dx^2),\label{consistL}\\
& \label{consistJ}
|J^{\alpha}[\phi]-J^{\alpha}_h[\phi]|
\leq K_I\Dx\begin{cases}
|D^2\phi|_0 & \text{when}\quad \gamma=[0,1),\\
(|D^2\phi|_0+|D^4\phi|_0) & \text{when}\quad \gamma=[1,2),
\end{cases}
\end{align}
for smooth bounded functions $\phi$ and where $\gamma\in(0,2)$ is defined in (A.4'). 
They are also monotone in the sense that they can be written as
\begin{align}\label{monL}
L^{\alpha}_h[\phi](t_n,x_{\bar\beta})=\sum_{\beta\in\Z^N}
l_{h,\beta,\bar\beta}^{\alp,n}\big[\phi(t_n,x_\beta)-\phi(t_n,x_{\bar\beta})\big]
\quad\text{with}\quad l_{h,\beta,{\bar\beta}}^{\alp,n}\geq0,\\ \label{monJ}
J^{\alpha}_h[\phi](t_n,x_{\bar\beta})=\sum_{\beta\in\Z^N}
j_{h,\beta,\bar\beta}^{\alp,n}\big[\phi(t_n,x_\beta)-\phi(t_n,x_{\bar\beta})\big]
\quad\text{with}\quad j_{h,\beta,{\bar\beta}}^{\alp,n}\geq0,
\end{align}
for any $\bar\beta\in\Z^N$ and $n\in\N_0$. 
We also assume without loss of generality that
$j_{\cdot,\beta,{\beta}}^{\cdot,\cdot}=0=l_{\cdot,\beta,{\beta}}^{\cdot,\cdot}$ 
for all $\beta\in\Z^N$. The sum \eqref{monL} is always 
finite, while the sum \eqref{monJ} is finite if
the L\'{e}vy measure $\nu$ is compactly supported. 
With $\gamma\in[0,2)$ defined in (A.4') and $\Dx<1$, we also have that
\begin{align}
&\bar l^{\alp,n}_{\bar\beta} := \sum_{\beta\in\Z^N}
l_{h,\beta,\bar\beta}^{\alp,n}\leq
K_l\sup_\alp\Big\{|a^\alp|_0\Dx^{-2}
+|b^\alp|_0\Dx^{-1}\Big\}\label{barL},\\
&\bar j^{\alp,n}_{\bar\beta}: =
\sum_{\beta\in\Z^N}
j_{h,\beta,\bar\beta}^{\alp,n}\leq K_j\Dx^{-1}.
\label{barJ}
\end{align}

From \eqref{consistL} and \eqref{consistJ} it 
immediately follows that the scheme \eqref{FD} is 
a consistent approximation of \eqref{eq:integro_pde}, with 
the truncation error bounded by
\begin{align}
\begin{split}
&\frac12|\phi_{tt}|_0\Dt +
\sup_{\alp,n}\Big\{|L^\alp[\phi]^n-L^\alp_h[\phi]^n|_0
+|J^\alp[\phi]^n-J^\alp_h[\phi]^n|_0\\
&\qquad+(1-\theta)|L^\alp[\phi]^{n-1}
-L^\alp[\phi]^n|_0+(1-\vartheta)|
J^\alp[\phi]^{n-1}-J^\alp[\phi]^n|_0\Big\},
\end{split}
\label{consist}
\end{align}
for smooth functions $\phi$. The last two terms are again bounded by
\begin{align}
\Dt \sup_{\alp}\Big\{|L^\alp[\phi_t]|_0
+|J^\alp[\phi_t]|_0\Big\}\leq
K\Dt\Big\{|\partial_t D\phi|_0+|\partial_t
D^2\phi|_0\Big\}. \label{eq:non_stand2}
\end{align}

Under a CFL condition, the scheme \eqref{FD} is also monotone, meaning that
there are numbers $b_{\beta,\tilde\beta}^{m,k}(\alp)\geq 0$ such that it can be written as
\begin{align}\label{def_mon}
\sup_\alp\Big\{b_{\bar\beta,\bar\beta}^{n,n}(\alp) U_{\bar\beta}^{n}
-\sum_{\beta\neq\bar\beta}b_{\bar\beta,\beta}^{n,n}(\alp)
U^{n}_\beta-\sum_{\beta} b_{\bar\beta,\beta}^{n,n-1}(\alp)
U_{\beta}^{n-1} -\Dt f^{n-1,\alp}_{\bar\beta}\Big\} =0,
\end{align}
for all $(x_{\bar\beta},t_{n})\in \mathcal G_h^+$.
From \eqref{monL} and \eqref{monJ}, we see that
\begin{align*}
&b_{\bar\beta,\bar\beta}^{n,m}(\alp)=\begin{cases}
1+\Dt\theta\,\bar l^{\alp,m}_{\bar\beta}+\Dt\vartheta\,
\bar j^{\alp,m}_{\bar\beta}&\text{when } m=n, \\
1-\Dt\Big[(1-\theta)\bar
l^{\alp,m}_{\bar\beta}+(1-\vartheta)\bar j^{\alp,m}_{\bar\beta}
-c^{\alp,m}_{\bar\beta}\Big] & \text{when } m=n-1,
\end{cases}\\
&b_{\bar\beta,\beta}^{n,m}(\alp)
=\begin{cases}
\Dt \theta
l_{h,\bar\beta,\beta}^{\alp,m}
+\Dt\vartheta j_{h,\bar\beta,\beta}^{\alp,m}&\text{when } m=n,\\
\Dt (1-\theta)
l_{h,\bar\beta,\beta}^{\alp,m}
+\Dt (1-\vartheta) j_{h,\bar\beta,\beta}^{\alp,m}\qquad\qquad &
\text{when } m=n-1,
\end{cases}
\end{align*}
where $\bar\beta\neq\beta$ and other choices of $m$ give zero. 
These coefficients are positive provided the following CFL condition holds:
\begin{align}\label{CFL}
\Dt\Big[(1-\theta)\bar
l^{\alp,m}_{\beta}+(1-\vartheta)\bar
j^{\alp,m}_{\beta}-c^{\alp,m}_{\beta}\Big]\leq
1\quad\text{for all}\quad\alp,\beta,m,
\end{align}
or alternatively by \eqref{barL} and \eqref{barJ}, if
$a^\alp\not\equiv0$, $c^\alp\geq0$,
$\Dx<1$,
\begin{align*}
\Dt\Big[(1-\theta)K_lC\Dx^{-2}+(1-\vartheta)K_j\Dx^{-1}\Big]\leq
1.
\end{align*}

Existence, uniqueness, and convergence results 
for the above approximation scheme are collected 
in the next theorem, while error estimates are 
postponed to Theorem \ref{ThmErr} in Section \ref{sec:ErrBnd}.

\begin{thm}
\label{FDthm}
Assume \ref{A1}--\ref{A3}, (A.4'), \eqref{consistL}--\eqref{barJ}, and
\eqref{CFL}. 
\smallskip

\noindent (a) There exists a unique bounded solution $U_h$ of \eqref{FD}--\eqref{FD_BC}.
\smallskip

\noindent (b) The scheme is
$L^\infty$-stable, i.e. $|U_h|\leq
e^{\sup_\alp|c^\alp|_0t_n} \big[|g|_0+t_n\sup_\alp|f^\alp|_0\big]$.

\smallskip
\noindent (c) $U_h$ converge uniformly to the viscosity solution $u$ 
of \eqref{eq:integro_pde}--\eqref{eq:bdcond} as $h\rightarrow 0$.
\end{thm}

\begin{proof}
The existence and uniqueness of bounded solutions follow by an
induction argument. Consider $t=t_n$ and assume $U^{n-1}$ is a given bounded
function. For $\eps>0$ we define the operator $T:U^{n}\rightarrow U^n$ by
$$
T U^n_\beta = U^n_\beta -\eps\cdot 
(\text{left hand side of \eqref{def_mon}})
\qquad\text{for all}\qquad\beta\in\Z^M.
$$
Note that the fixed point equation $U^n=TU^n$ is equivalent to
equation \eqref{FD}. Moreover, for sufficiently 
small $\eps$, $T$ is a contraction operator
on the Banach space of bounded functions on $\Dx\,\Z^N$ under the
$\sup$-norm. Existence and uniqueness then follows from the fixed
point theorem (for $U^n$) and for all of $U$ by induction since
$U^0=g|_{\mathcal{G}_h^0}$ is bounded.

To see that $T$ is a contraction we use the definition 
and sign of the $b$-coefficients:
\begin{align*}
&TU^n_\beta-T\tilde U^n_\beta\\
&\leq \sup_\alp\Big\{\Big[1-\eps[1+\Dt(\theta\bar
l^{\alp,n}_\beta+\vartheta\bar j^{\alp,n}_\beta)]\Big](U^n_\beta-\tilde
U^n_\beta)+\eps\Dt(\theta\bar l^{\alp,n}_\beta+\vartheta\bar
j^{\alp,n}_\beta)|U^n_\cdot-\tilde U^n_\cdot|_0\Big\}\\
&\leq (1-\eps)|U^n_\cdot-\tilde U^n_\cdot|_0,
\end{align*}
provided $1-\eps(1+\Dt(\theta\bar l^{\alp,n}_\beta 
+\vartheta\bar j^{\alp,n}_\beta))]\geq0$ for all $\alp,\beta,n$.
Taking the supremum over all $\beta$ and interchanging the role of $U$
and $\tilde U$ proves that $T$ is a contraction.

Much the same argument, utilizing \eqref{def_mon}, establishes 
that $U_h$ is bounded by a constant independent of $h$:
$$
|U^n|_0
\leq (1+\Dt\sup_\alp|c^\alp|_0)^n\Big[|g|_0
+n\Dt\sup_\alp|f^\alp|_0\Big]
\leq e^{\sup_\alp|c^\alp|_0t_n}
\Big[|g|_0+t_n\sup_\alp|f^\alp|_0\Big].
$$
In view of this bound, the convergence of $U_h$ to the solution $u$ of 
\eqref{eq:integro_pde}--\eqref{eq:bdcond} follows by adapting 
the Barles-Souganidis argument \cite{Barles:1991tc} to the present non-local context. 
Alternatively, convergence follows from Theorem \ref{ThmErr} if we also assume (A.5).
\end{proof}

\begin{rem}$\quad$
\smallskip

a. One suitable choice of $J^\alp_h$ will be derived in Section
\ref{sec:ApprInt}, while for $L^\alp_h$ there are several choices
that satisfies \eqref{consistL} and \eqref{monL}, e.g., the scheme by
Bonnans and Zidani \cite{BZ} or the (standard) schemes of Kushner
\cite{Kushner:1992mq}. In Appendix \ref{sec:Lh} we show that one of the
schemes of Kushner fall into our framework if $a^\alp$ is diagonally dominant.

\smallskip

b. For the differential part, the choices $\theta=0 ,\ 1$, and $1/2$ give
explicit, implicit, and Crank-Nicholson discretizations. When
$\vartheta>0$, the integral term is evaluated implicitly. This leads
to linear systems with full matrices and is not 
used much in the literature.

\smallskip

c. By parabolic regularity $D^2\sim \partial_t$ and
\eqref{eq:non_stand2} is similar to $\Dt|\phi_{tt}|_0$. When
$\theta=1/2=\vartheta$ the scheme \eqref{FD} (Crank-Nicholson!) is
second order in time $\mathcal O(\Dt^2)$ and \eqref{consist} is no
longer optimal.

\smallskip

d. When $\gamma=0$ the leading error term in $J_h[u]$ (see
\eqref{consistJ}) comes from
difference approximation of the term $Du \int\eta \nu $. This
difference approximation also give rise the term $\Dx^{-1}$ in \eqref{barJ}.
\smallskip

\end{rem}

\section{Error estimates for general monotone approximations}
\label{sec:ErrBnd}

In this section we present error estimates for nonlinear
general monotone approximation schemes for IPDEs. 
As a corollary we obtain an error estimate for the scheme \eqref{FD}--\eqref{FD_BC} 
defined in Section \ref{sec:FD}. These results, which extend those in \cite{Barles:2006jf} to 
the non-local IPDE context, can be applied to ``any'' L\'{e}vy-type 
integro operator. Earlier results apply to either linear problems, 
specific schemes, or restricted types of L\'{e}vy operators, see
\cite{Schwab2,CV,Jakobsen:2005tp,Biswas:2006tr}. In particular, previous error estimates
do not apply to the approximation scheme \eqref{FD}.

Let us write \eqref{eq:integro_pde} as $u_t+F[u]=0$ where
$F[u]:=F(t,x,u,Du,D^2u,u(t,\cdot))$ denotes the $\sup$ 
part of \eqref{eq:integro_pde}. We write approximations of $u_t+F[u]=0$ as
\begin{align}\label{abs_scheme}
	S(h,t,x,u_h(t,x),[u_h]_{t,x}) &=0 \quad
	\text{in}\quad \mathcal{G}_h^+,\\
	\label{BC_scheme}
	u_h(0,x) &= g_h(x) \quad \text{in}
	\quad \mathcal{G}_h^0,
\end{align}
where $S$ is the approximation of \eqref{eq:integro_pde} defined on the mesh
$\mathcal{G}_h\subset \bar{Q}_T$ with ``mesh'' parameter 
$h=(\Delta t,\Delta x, \Delta z)$ (time, space, quadrature parameters). The
solution is typified by $u_h$ and by $[u_h]_{t,x}$ we denote a
function defined at $(t,x)$ in terms of the values taken by $u_h$
evaluated at points other than $(t,x)$. Note that the grid does not
have to be uniform or even discrete. 

We assume that \eqref{abs_scheme} satisfies the following set of
(very weak) assumptions:
\begin{Assumptions}
	\item[{\bf (S1)}]\label{S1}{\bf(Monotonicity)} There exist 
	$\lambda, \mu\ge 0$, $h_0> 0$ such that, if $|h|\le h_0$, $u\le v$ 
	are functions in $C_b(\mathcal{G}_h)$ and 
	$\phi(t)= e^{\mu t}(a+bt)+c$ for $a,b,c\ge 0$, then
	\begin{align*} 
		S(h,t,x,r+\phi(t),[u+\phi]_{t,x})\ge
		S(h,t,x,r,[v]_{t,x})+\frac{b}{2}-\lambda
		c\quad\textrm{in}~\mathcal{G}_h^+.
	\end{align*}

	\item[{\bf(S2)}]\label{S2}{\bf(Regularity)} For each $h$ and 
	$\phi\in C_b(\mathcal{G}_h)$, the mapping
	$$
	(t,x)\mapsto S\big(h,t,x,\phi(t,x),[\phi]_{t,x}\big)
	$$
	is bounded and continuous in $\mathcal{G}_h^+$ and 
	the function $r\mapsto S(h,t,x, r,[\phi]_{t,x})$ 
	is uniformly continuous for bounded $r$, uniformly in $t,x$.
	
	\smallskip

	\item[{\bf (S3)}]\label{S3i}{\bf(i) (Sub-consistency)} 
	There exists a function $E_1(\tilde{K}, h, \epsilon)$ such 
	that, for any sequence $\{\phi_\epsilon\}_\epsilon$ of 
	smooth bounded functions satisfying
	\begin{align*}
		|\partial_t^{\beta_0} D^{\beta^\prime}\phi_\epsilon|\le
		\tilde{K}\epsilon^{1-2\beta_0-|\beta^\prime|}
		\quad\textrm{in}~\bar{Q}_T,\quad \textrm{for any}~
		\beta_0\in\mathbb{N},\beta^{\prime}\in\mathbb{N}^N,
	\end{align*} where 
	$|\beta^\prime|=\sum_{i=1}^{N}\beta_i^{\prime}$, the following 
	inequality holds in $\mathcal{G}_h^+$:
	\begin{align*}
		S\big(h,t,x,\phi_\epsilon(t,x),
		[\phi_\epsilon]_{t,x}\big)\le \phi_{\epsilon t}
		+F(t,x,\phi_\epsilon,D\phi_\epsilon,D^2\phi,\phi_\epsilon(t,\cdot))
		+E_1(\tilde{K},h, \epsilon).
	\end{align*}

	\item[{\bf (S3)}]\label{s3ii}{\bf (ii) (Super-consistency)} There exists a function 
	$E_2(\tilde{K},h,\epsilon)$ such that, for any sequence 
	$\{\phi_\epsilon\}_\epsilon$ of smooth bounded functions satisfying 
	\begin{align*}
		|\partial_t^{\beta_0} D^{\beta^\prime}\phi_\epsilon|
		\le \tilde{K}\epsilon^{1-2\beta_0-|\beta^\prime|}
		\quad\textrm{in}~\bar{Q}_T,\quad 
		\textrm{for any}~\beta_0\in\mathbb{N},\beta^{\prime}\in\mathbb{N}^N,
	\end{align*} 
	the following inequality holds in $\mathcal{G}_h^+$:
	\begin{align*}
		S\big(h,t,x,\phi_\epsilon(t,x), [\phi_\epsilon]_{t,x}\big)\ge 
		\partial_t \phi_{\epsilon}
		+F(t,x,\phi_\epsilon, D\phi_\epsilon,D^2\phi,\phi_\epsilon(t,\cdot))
		-E_2(\tilde{K}, h, \epsilon).
	\end{align*}
\end{Assumptions}

\begin{rem} 
In (S3), we typically take $\phi_\epsilon=w_\epsilon*\rho_\epsilon$ for 
some sequence $(w_\epsilon)_\epsilon$ of uniformly bounded and Lipschitz continuous 
functions, and $\rho_\epsilon$ is the mollifier defined in Section \ref{sec:intro}.
\end{rem}

\begin{rem} Assumption (S1) implies monotonicity in
$[u]$ (take $\phi = 0$), and parabolicity of the scheme
\eqref{abs_scheme} (take $u= v$). This last point is easier to
understand from the following more restrictive assumption:

\begin{Assumptions}
\item[\bf(S1')](Monotonicity) There exist $\lambda \ge 0,
\bar{K}> 0$ such that if $u\le v; u,v\in C_b(\mathcal{G}_h)$ and 
$\phi:[0,T]\rightarrow \mathbb{R}$ smooth, then
\begin{align*}
	&S(h,t,x,r+\phi(t),[u+\phi]_{t,x})\\
	&\ge S(h,t,x,r,[v]_{t,x})+\phi^\prime(t)
	-\bar{K}\Delta t|\phi^{\prime\prime}(t)|_0
	-\lambda \phi^{+}(t).
\end{align*}
\end{Assumptions} 
It is easy to see that (S1') implies (S1), cf. \cite{Barles:2006jf}.
\end{rem}

The main consequence of (S1) and (S2) is the following comparison principle satisfied
by scheme \eqref{abs_scheme} (for a proof, cf.~\cite{Barles:2006jf}):
\begin{lem}\label{cmpare_scm}
Assume (S1), (S2), $g_1, g_2\in C_b(\mathcal{G}_h)$, and
$u,v\in C_b(\mathcal{G}_h)$ satisfy
\begin{align*}
	S(h,t,x,u(t,x),[u]_{t,x})\le g_1\quad\text{and}\quad
	S(h,t,x,v(t,x),[v]_{t,x})\ge g_2\quad
	\textrm{in}\quad\mathcal{G}_h^+.
\end{align*} 
Then, for $\lambda$ and $\mu$ as in (S1),
\begin{align*}
	u-v\le e^{\mu t}|(u(0,\cdot)-v(0,\cdot))^+
	|_0+2te^{t}|(g_1-g_2)^+|_0.
\end{align*}
\end{lem}

The following theorem is our first main result.

\begin{thm}[Error Estimate]\label{err_est}
Assume \ref{A1}--\ref{A4}, (S1), (S2) hold, and that the 
approximation scheme \eqref{abs_scheme}--\eqref{BC_scheme} has
a unique solution $u_h\in C_b(\mathcal{G}_h)$, for each sufficiently small $h$. 
Let $u$ be the exact solution of \eqref{eq:integro_pde}--\eqref{eq:bdcond}.

\begin{itemize}
\item[ a)]{\bf (Upper Bound)} If (S3)(i) holds, then there
exists a constant $C$, depending only on $\mu, K$ in (S1) and \ref{A2}, such that
\begin{align*}
	u-u_h\le e^{\mu t}|(g-g_{h})^+|_0
	+C\min_{\epsilon > 0}\big(\epsilon
	+E_1(|u|_1,h,\epsilon)\big)
	\quad\textrm{in}~\mathcal{G}_h.
\end{align*}

\item[ b)]{\bf (Lower Bound)} If (S3)(ii) holds, then
there exists a constant $C$, depending only
on $\mu, K$ in (S1) and \ref{A2}, such that
\begin{align*}
	u-u_h\ge -e^{\mu t}|(g-g_{h})^-|_0
	-C\min_{\epsilon > 0}
	\big(\epsilon^{\frac 13}
	+E_2(|u|_1,h,\epsilon)\big)
	\quad\textrm{in}~\mathcal{G}_h.
\end{align*}
\end{itemize}
\end{thm}

We prove this theorem in Section \ref{sec:PfErr}.

\begin{rem}
Theorem \ref{err_est} applies to all L\'{e}vy type 
non-local operators. Note that the lower bound is
worse than the upper bound, and may not be optimal. In
certain special cases it is possible to prove better bounds, however 
until now such results could only be obtained in the non-degenerate
linear case \cite{Schwab2,CV} or under very strong restrictions on the
non-local term \cite{Biswas:2006tr,CJ:08}. More information
on such non-symmetric error bounds can be found in \cite{Barles:2006jf}.
\end{rem}

\begin{rem}
For a finite difference-quadrature type discretization of
\eqref{eq:integro_pde}, the truncation error would typically look like
\begin{align*}
	&|\phi_t+F(t,x,\phi,D\phi, D^2\phi,\phi(t,\cdot))-S(h,t,x,\phi(t,x),[\phi]_{t,x})|\\ 
	&\le K\sum_{\beta_0}|\partial_t^{\beta_0^0}D^{\beta_0^\prime}\phi|_0
	\Delta t^{k_{\beta_0}}+K\sum_{\beta_1}|\partial_t^{\beta_1^0}
	D^{\beta_1^\prime}\phi|_0{\Delta x}^{k_{\beta_1}}
	+K\sum_{\beta_2}|\partial_t^{\beta_2^0}
	D^{\beta_2^\prime}\phi|_0{\Delta z}^{k_{\beta_2}},
\end{align*}
where $\beta_0=(\beta_0^0,\beta_0')$, $\beta_1=(\beta_1^0,\beta_1')$, 
$\beta_2=(\beta_2^0,\beta_2')$ are 
multi-indices and $k_{\beta_0},k_{\beta_2},k_{\beta_2}$ are real numbers. 
In this case, the function $E$ in (S3) is obtained by taking $\phi:=\phi_\epsilon$ 
in the above inequality:
\begin{align*}
	&E_1=E_2= \tilde{K}K\sum_{\beta_0,\beta_1,\beta_2}
	\Big[\epsilon^{1-2\beta_0^0-|\beta_0^\prime|}
	\Delta t^{k_{\beta_0}}+\epsilon^{1-2\beta_1^0
	-|\beta_1^\prime|}{\Delta x}^{k_{\beta_1}}
	+\epsilon^{1-2\beta_2^0-|\beta_2^\prime|}{\Delta z}^{k_{\beta_2}}\Big].
\end{align*}
An optimization with respect to $\eps$ yields the final convergence rate. 
Observe that the obtained rate reflects a potential 
lack of smoothness in the solution.
\end{rem}

We shall now use Theorem \ref{err_est} to prove error estimates for 
the finite difference-quadrature scheme \eqref{FD}.

\begin{thm}
\label{ThmErr}
Assume \ref{A1}--\ref{A3}, (A.4'), (A.5), \eqref{consistL}--\eqref{barJ}, \eqref{CFL} 
hold, and that $u$ and $U_h$ are the solutions respectively 
of \eqref{eq:integro_pde}--\eqref{eq:bdcond} and \eqref{FD}--\eqref{FD_BC}.

There are constants $K_L, K_J\geq0$, $\delta>0$ such 
that if $\Dx\in(0,\delta)$ and $\Dt$ satisfies 
the CFL condition \eqref{CFL}, then in $\mathcal G_h$,
\begin{align*}
\begin{array}{cl}
-K(\Dt^{1/ 10}+\Dx^{1/5}) \leq u-U_h\leq
K(\Dt^{1/4}+\Dx^{1/2})&\text{for}\quad \gamma\in[0,1),\\
&\\ -K(\Dt^{1/10}+\Dx^{1/10}) \leq u-U_h\leq
K(\Dt^{1/4}+\Dx^{1/4})&\text{for}\quad \gamma\in[1,2).
\end{array}
\end{align*}
\end{thm}

\begin{proof}
Let us write the scheme \eqref{FD} in abstract form
\eqref{abs_scheme}. To this end, set $[u]_{t,x}(s,y)= u(t+s,x+y)$ 
and divide \eqref{def_mon} by $\Dt$ to see that \eqref{FD} 
takes the form \eqref{abs_scheme} with
\begin{align*}
	S(h, t_n,x_\beta, r, [u]_{t_n,x_\beta}) 
	=\sup_{\alpha\in\mathcal{A}}\Bigg\{&
	\frac{b_{\beta,\beta}^{n,n}(\alp)}{\Delta t}
	r-\sum_{\bar\beta\neq\beta}
	\frac{b_{\beta,\bar\beta}^{n,n}(\alp)}{\Delta t}[u]_{t_n,x_\beta}
	(0,x_{\bar\beta}-x_\beta)
	\\ &\qquad -\sum_{\bar\beta}
	\frac{b_{\beta,\bar\beta}^{n,n-1}(\alp)}{\Delta t}[u]_{t_n,x_\beta}
	(-\Dt,x_{\bar\beta}-x_\beta)\Bigg\}.
\end{align*}
By its definition \eqref{FD}, monotonicity \eqref{def_mon}, and
consistency \eqref{consist}, this scheme obviously
satisfies assumptions (S1) -- (S3) if the
CFL condition \eqref{CFL} holds. In particular, from
\eqref{consist} and \eqref{consistL}, \eqref{consistJ},
\eqref{eq:non_stand2}, we find that
\begin{align*}
	E_1(\tilde K,h,\eps)=E_2(\tilde K,h,\eps) 
	=\begin{cases}
		C\tilde K(\Dt\eps^{-3}+\Dx\eps^{-1}+\Dx^{2}\eps^{-3}), &\gamma\in[0,1) \\
		C\tilde K(\Dt\eps^{-3}+\Dx\eps^{-1}+(\Dx^{2}+\Dx)\eps^{-3}), &\gamma\in[1,2).
	\end{cases}
\end{align*}
The result then follows from Theorem \ref{err_est} and a minimization
with respect to $\eps$.
\end{proof}

\begin{rem}
The error estimate is independent of $\gamma$ and robust in the sense
that it applies to non-smooth solutions.
\end{rem}

\section{New approximations of the non-local term}
\label{sec:ApprInt}

In this section we derive direct approximations $J^\alp_h[u]$ of 
the non-local integro term $J^\alp[u]$ appearing in \eqref{eq:integro_pde}. 
As in \cite{AO07} (cf.~also \cite{Lynch:2003yq}), the 
idea is to perform integration by parts to reduce 
the singularity of the measure. For the full discretization of \eqref{eq:integro_pde} 
along with convergence analysis, we refer to Section \ref{sec:FD}.

We consider 3 cases
separately: (i) $\int_{|z|<1}\nu(dz)<\infty$, (ii)
$\int_{|z|<1}|z|\nu(dz)<\infty$, and (iii)
$\int_{|z|<1}|z|^2\nu(dz)<\infty$.
Note that in cases (i) and (ii) we
can write the non-local
operator in the form
\begin{align}
\label{decomp}
J^\alp[\phi](t,x)= I^\alp[\phi](t,x)
- \bar b^\alpha(x) D\phi,
\end{align}
where
\begin{align*}
	I^\alp[\phi](t,x) &
	:=\int_{|z|>0}\Big(\phi(t,x+\eta^\alpha(t,x,z))-\phi\Big)\nu(dz),\\ 
	\bar b^{\alpha}(x) & :=\int_{0<|z|<1}\eta^{\alpha}(t,x,z)\nu(dz),
\end{align*}
for smooth bounded functions $\phi$. The reason is that $I^\alp[\phi]$
and $\bar b^\alpha(x)$ are well-defined under assumptions (A.2),
(A.3), (A.4) if either (i) or (ii) holds. 
Furthermore, $\bar b^\alpha(x)$ will be bounded and $x$-Lipschitz. 
The term $\bar b^\alpha D\phi$ will be approximated by quadrature and
upwind finite differences as in Appendix \ref{sec:Lh} leading to a
first order method. We skip the standard details and focus on the
non-local term $I^\alp[\phi]$.

To simplify the presentation a bit, we will only consider the Cartesian $x$-grid 
$\{x_\beta\}_\beta=\Delta x\, \mathbb{Z}^N$, but it is possible 
to consider unstructured non-degenerate families of grids. On our
grid we define a positive and 2nd order interpolation 
operator $i_h$, i.e., an operator satisfying
\begin{align}
	& i_{h}\phi(x)= \sum_{\beta\in \mathbb{Z}^N} w_\beta(x)\phi(x_\beta)
	\quad \text{with}\quad w_\beta(x)\geq0,\label{eq:lint} \\
	&|E_I[\phi](x)|:=|\phi(x)-i_h\phi(x)|\leq K_I\Delta x^{2}
	|D^2\phi|_0,\label{int_error_bnd}
\end{align}
for all $x\in\R^N$ and where $w_\beta(x)\geq0$ are basis functions
satisfying $w_\beta(x_{\bar\beta})=\delta_{\beta,\bar\beta}$ and $\sum_\beta
w_\beta\equiv1$. Linear and multi-linear interpolation satisfy these
assumptions. Note that higher order interpolation is not monotone in general.

We will also need the following monotone difference operators:
\begin{align}
	\delta_{r,h}^{\pm} \phi(r,y) &= \pm\frac{1}{\Delta x}
	\big\{\phi(r\pm\Delta x,y)-\phi(r,y)\big\},\label{delta}\\
	\Delta_{rr,k} \phi(r,y) &= \frac{1}{k^2}
	\big\{\phi(r+k,y)-2\phi(r,y)+\phi(r- k,y)\big\},\label{Delta}
\end{align}
for functions $\phi(r,y)$ on $\R\times\R^K$ for 
some $K\in \N$. For smooth $\phi$ we have
$$
|\delta_{r,h}^{\pm} \phi-\partial_r\phi|\leq \frac12
|\phi_{rr}|_0\Dx,\qquad |\Delta_{rr,k}
\phi-\partial_r^2\phi|\leq 
\frac1{12} |\partial_r^4\phi|_0|k|^2.
$$

\subsection{Finite L\'{e}vy measures} 
Assuming $\int_{|z|<1}\nu(dz)<\infty$, we 
approximate the term $I^\alp[\phi]$ defined in \eqref{decomp} by
\begin{align*}
	I^\alp_h[\phi](t,x)=
	Q_h\Big[(i_h\phi)(t,x+\eta^\alpha(t,x,z))-\phi(t,x)\Big],
\end{align*}
where $Q_h$ denotes a positive quadrature rule on the $z$-grid
$\{z_\beta\}_\beta\subset \R^M$ with maximal grid spacing $\Dz$, satisfying
\begin{align*}
	& Q_{h}[\phi]= \sum_{\beta\in \mathbb{Z}^M} \omega_\beta \phi(z_\beta)
	\quad \text{with}\quad\omega_\beta\geq0,\\
	&|E_Q[\phi]|:=|\textstyle{\int} \phi(z)\nu(dz)-Q_h[\phi]|
	\le K_Q\Delta z^{k_Q} |D^{k_Q}\phi|_0\textstyle{\int}\nu(dz),
\end{align*}
for smooth bounded functions $\phi$, where $K_Q\geq0$ and $k_Q\in\N$.
Many quadrature methods satisfies these
requirements, e.g., compound Newton-Cotes methods of order less than
9 and Gauss methods of arbitrary order. Note that the $z$-grid does
not have to be a Cartesian grid. This method is at most 2nd
order accurate because
\begin{align*}
	&I^\alp[\phi]
	=I^\alp_h[\phi] + E_I[\phi(\cdot,\cdot+\eta^\alpha)]\textstyle{\int}\nu(dz)
	+ E_Q[\phi(\cdot,\cdot+\eta^\alpha)-\phi],
\end{align*}
and it is monotone by construction, satisfying \eqref{monJ} and
\eqref{barJ}. The $\mathcal O(\Dx^{-1})$ term in \eqref{barJ} comes from
the discretization of the $\bar b^\alp$ term in \eqref{decomp}.

\subsection{Unbounded L\'{e}vy measures I} Now we assume that
$\int_{|z|<1}|z|\nu(dz)<\infty$, or more precisely that (A.4') holds
with $\gamma<1$. We consider the one-dimensional and
multi-dimensional cases separately.

\subsubsection{One-dimensional case $(M=1)$.}\label{sec:11}
Now $I^\alpha[\phi]$ in \eqref{decomp} takes the form
\begin{align*}
	I^\alpha[\phi](t,x) &=\int_{\mathbb{R}\setminus\{0\}}
	\big[\phi(t,x+\eta^\alpha(t,x,z))-\phi(t,x)\big]k(z)dz.
\end{align*}
We approximate this term by
\begin{align}\label{I-1}
	\begin{split}
		&I_h^{\alpha}[\phi](t,x)\\
		&=\sum_{n= 0}^\infty \Big[\delta^+_{z,h}
		(i_h\phi)(t,x+\eta^\alpha(t,x,z_n))k_{h,n}^+
		-\delta^-_{z,h}(i_h\phi)(t,x+\eta^\alpha(t,x,z_{-n})) 
		k_{h,n}^-\Big],
	\end{split}
\end{align}
where $z_n=n\,\Dx$, $\delta^\pm_{z,h}$ is defined in
\eqref{delta}, the $x$-interpolation $i_h$ satisfies \eqref{eq:lint}
and \eqref{int_error_bnd}. Moreover,
\begin{align*}
	\begin{cases}
		k_{h,n}^+:=
		\int_{z_n}^{z_{n+1}}\hat{k}(z)dz,\\
		k_{h,n}^-:=\int^{z_{-n}}_{z_{-(n+1)}}\hat{k}(z)dz
	\end{cases} \quad\text{and}\quad
	\hat{k}(z):=
	\begin{cases}
		\int_{-\infty}^zk(\zeta)\, d\zeta, &\text{if}\quad z<0,\\
		\int_{z}^\infty k(\zeta)\, d\zeta, &\text{if}\quad z>0.
	\end{cases}
\end{align*}
By (A.4') ($M=1$ and $\gamma<1$), $0\leq\int_\mathbb{R}\hat{k}(z)dz<\infty$.

To derive this approximation, the key idea is 
to perform integration by parts:
\begin{align*}
	&I^\alpha[\phi](t,x)=\Big(\int_{-\infty}^0+\int_{0}^\infty\Big)
	\big(\phi(t,x+\eta^\alpha(t,x,z))-\phi(t,x)\big)k(z)dz\\ 
	& = \int_{0}^\infty \frac{\partial}{\partial z}
	\big(\phi(t,x+\eta^\alpha(t,x,z))\big)\hat{k}(z)dz
	-\int_{-\infty}^0 \frac{\partial}{\partial z}
	\big(\phi(t,x+\eta^\alpha(t,x,z))\big)\hat{k}(z)dz,
\end{align*}
for bounded $C^1$ functions $\phi$. Write $I^\alpha[\phi]=I^{\alpha,+}[\phi]+I^{\alpha,-}[\phi]$, 
and use quadrature, finite differencing, and interpolation to proceed as follows:
\begin{align*}
	I^{\alpha, +}[\phi](t,x) & :=
	\int_{0}^\infty \partial_z\big[\phi(t,x+\eta^\alpha(t,x,z))\big]\hat{k}(z) dz\\
	& \simeq \sum_{n= 0}^\infty
	\partial_z\big[\phi(t,x+\eta^\alpha(t,x,z))\big]\big|_{z=z_n} k_{h,n}^+\\
	&\simeq \sum_{n= 0}^\infty
	\frac{\phi(t,x+\eta^\alpha(t,x,z_{n}+\Dx))
	-\phi(t,x+\eta^\alpha(t,x,z_n))}{\Delta x}k_{h,n}^+\\
	& \simeq \sum_{n= 0}^\infty
	\frac{(i_h\phi)(t,x+\eta^\alpha(t,x,z_{n}+\Dx))
	-(i_h\phi)(t,x+\eta^\alpha(t,x,z_n))}{\Delta x}k_{h,n}^+.
\end{align*}
In a similar way we can discretize $I^{\alpha,-}[\phi]$ and \eqref{I-1} follows.

The approximation just proposed is consistent since
$$
I^{\alpha}[\phi](t,x)= I_h^{\alpha}[\phi](t,x) 
+ E_Q+ E_{\mathrm{FDM}}+E_I,
$$
where $E_Q$, $E_{\mathrm{FDM}}$, and $E_I$ denote respectively the
error contributions from the approximation of the integral (1st
order), the difference approximation (up-winding, 1st order), and the
2nd order interpolation. These terms can be estimated as follows:
\begin{align*}
	&|E_Q|\le \Delta x \,|\partial_z^2\phi(\cdot+\eta^\alp)|_0
	\int_{\R} \hat k(z) dz,\\
	&|E_{\mathrm{FDM}}|\leq \frac12 \Dx\, |\partial_z^2\phi(\cdot+\eta^\alp)|_0
	\int_{\R} \hat k(z) dz, \\
	&|E_I|\leq 2 \Delta x\, |D_x^2\phi(\cdot+\eta^\alp)|_0 
	\int_{\R} \hat k(z) dz.
\end{align*}

The discretization \eqref{I-1} is also monotone satisfying
\eqref{monJ} and also \eqref{barJ} (when $I_h^\alp$ replaces $J_h^\alp $).
To see this, note that $i_h\phi(x_{\bar\beta})=\phi(x_{\bar\beta})$ and that by
(A.2) $\eta(t,x,0)=0$. Hence we can reorganize the 
sum defining $I_h^{\alp,+}$ and write
\begin{align*}
	& I_h^{\alp,+}[\phi](t,x_{\bar\beta})\\
	& \quad = -\frac{1}{\Delta x}k_{h,0}^+
	\phi(t,x_{\bar\beta})+\frac{1}{\Delta x}
	\sum_{n=1}^\infty(k_{h,n-1}^+
	-k_{h,n}^+)(i_h\phi)(t,x+\eta^\alpha(t,x_{\bar\beta},z_n))\\
	& \quad =\frac{1}{\Delta x}\sum_{n=1}^\infty(k_{h,n-1}^+-k_{h,n}^+)
	\big[(i_h\phi)(t,x_{\bar\beta}+\eta^\alpha(t,x_{\bar\beta},z_n))
	-\phi(t,x_{\bar\beta})\big].
\end{align*}

In a similar way
$$
I_h^{\alp,-}[\phi](t,x_{\bar\beta}) 
= \frac{1}{\Delta x}\sum_{n=1}^\infty(k_{h,n-1}^--k_{h,n}^-)
\big[(i_h\phi)(t,x_{\bar\beta}+\eta^\alpha(t,x_{\bar\beta},z_{-n}))
-\phi(t,x_{\bar\beta})\big]. 
$$

Since $\hat k$ is increasing on 
$(0,\infty)$ and decreasing on $(-\infty,0)$,
\begin{equation*}
	k_{h,n-1}^\pm> k_{h,n}^\pm,
\end{equation*}
and hence by \eqref{eq:lint} and $\sum_\beta w_\beta\equiv1$,
\eqref{monJ} and \eqref{barJ} hold with
\begin{align*}
&j_{h, \beta,{\bar\beta}}^{\alp,n}=\frac1{\Dx}
\sum_{l\in\Z\setminus\{0\}}
w_{\beta}(x_{\bar\beta}+\eta^\alp(t_n,x_{\bar\beta},z_l))(k_{h,|l|-
1}^{\mathrm{sign}(l)}-k_{h,|l|}^{\mathrm{sign}(l)})\geq 0,\\
&\bar j_{\bar\beta}^{\alp,n}=\frac1{\Dx}\sum_{
l\in\Z\setminus\{0\}}(k_{h,|l|-
1}^{\mathrm{sign}(l)}-k_{h,|l|}^{\mathrm{sign}(l)})\sum_\beta
w_{\beta}(x_{\bar\beta}+\eta^\alp(t_n,x_{\bar\beta},z_l))
=\frac{k_{h,0}^++k_{h,0}^-}{\Dx},
\end{align*}
and $k_{h,0}^\pm=\mathcal O(\Dx^{1-\gamma})$. The leading 
$\mathcal O(\Dx^{-1})$ term in \eqref{barJ} comes from discretizing 
the $\bar b^\alp$ term in \eqref{decomp}.

\subsubsection{Multi-dimensional case $(M>1)$.}\label{sec:1M}
In this case we write $I_h^\alp[\phi]$ of \eqref{decomp} in polar
coordinates and propose the following approximation:
\begin{equation}\label{Ih1M}
	I^\alpha_h[\phi](t,x) =\int_{|y|= 1}\sum_{n= 0}^\infty
	\delta_{r,h}^+\big[i_h\phi(t,x+\eta^\alpha(t,x,r_ny))\big]
	k_{h,n}(y)dS_y,
\end{equation}
where $r_n=n\,\Dx$, $dS_y$ is the surface measure on the unit sphere in
$\R^M$, $\delta^\pm_{z,h}$ is defined in
\eqref{delta}, the $x$-interpolation $i_h$ satisfies \eqref{eq:lint}
and \eqref{int_error_bnd}. Moreover,
\begin{align*}
	k_{h,n}(y)
	=\int_{r_n}^{r_{n+1}} \hat{k}(r,y) dr
	\qquad\text{and}\qquad
	\hat{k}(r,y)= \int_r^\infty k(sy)s^{M-1}ds.
\end{align*}
By assumption (A.4') with $\gamma\in(0,1)$, $0\leq\int_0^\infty\hat k(r,y) dr\leq
C<\infty$ for all $|y|=1$.

To derive this approximation we use polar coordinates and integrate by
parts in the radial direction. Let $\phi$ be a bounded $C^1$ function, and set
$$
G^\alp(t,x,z):=\phi(t,x+\eta^\alpha(t,x,z))-\phi(t,x).
$$
Then
\begin{align*}
	I^\alpha[\phi](t,x) & = 
	\int_{\mathbb{R}^M \setminus\{0\}}G^\alp(t,x,z)k(z) dz\\
	& = \int_{|y|= 1}\big[\int_0^\infty G^\alp(t,x,ry) r^{M-1}k(ry) dr\big] dS_y
	\\ & = \int_{|y|= 1}\big[\int_0^\infty
	\frac{\partial }{\partial r}G^\alp(t,x,ry) \hat{k}(r,y) dr\big] dS_y,
\end{align*}
and \eqref{Ih1M} follows by discretizing the inner integral as in
Section \ref{sec:11}.

This is a consistent first order 
approximation of $I^{\alpha}[\phi]$ since
\begin{align*}
	I^{\alpha}[\phi](t,x)= I_h^{\alpha}[\phi](t,x) 
	+ E_Q+ E_{\mathrm{FDM}}+E_I,
\end{align*}
where $E_Q$, $E_{\mathrm{FDM}}$, $E_I$ have the same meaning 
as in Section \ref{sec:11}, and these terms can be estimated as follows:
\begin{align*}
	&|E_Q|\le \Delta x \,|D_z^2\phi(\cdot+\eta^\alp)|_0M_{\hat k},\\
	&|E_{\mathrm{FDM}}|\leq \frac12 \Dx\,
	|D_z^2\phi(\cdot+\eta^\alp)|_0M_{\hat k}, \\
	&|E_I|\leq \Delta x\, |D_x^2\phi(\cdot+\eta^\alp)|_0 M_{\hat k},
\end{align*}
where $M_{\hat k}=\int_{|y|=1}\int_0^\infty \hat k(r,y) dr\, dS_y$. 

The approximation $I^{\alpha}_h[\phi]$ is also monotone, 
and satisfies \eqref{monJ} and \eqref{barJ}. 
This follows as in Section \ref{sec:11}, since
$I^\alpha_h[\phi](t,x_{\bar\beta})$ can be written as
\begin{align*}
	\frac{1}{\Delta x}\int_{|y|= 1} 
	\sum_{n=1}^\infty\big[k_{h,n-1}(y)-k_{h,n}(y)\big]
	\Big[(i_h\phi)(t,x_{\bar\beta}+\eta^\alpha(t,x_{\bar\beta},r_ny))
	-\phi(t,x_{\bar\beta})\Big]dS_y,
\end{align*}
where for fixed $y$, $k_{h,n}(y)$ is a decreasing function in $n$
since $\hat k(r,y)$ decreasing in $r$. Moreover, $\bar
l^{\alp,n}_\beta$ has a term like $\frac{1}{\Dx}k_{h,0}(y)=\mathcal
O(\Dx^{-\gamma})$ plus the leading $\mathcal O(\Dx^{-1})$ term which
comes from the discretization of the $\bar b^\alp$ term in \eqref{decomp}.

\subsection{Unbounded L\'{e}vy measures II}\label{g12}
We assume that $\int_{|z|<1}|z|^2\nu(z)dx<\infty$, or more precisely
that (A.4') hold with $\gamma\in[1,2)$. In this case the 
decomposition \eqref{decomp} is not valid. Again, we consider the one-dimensional and 
multi-dimensional cases separately.

\subsubsection{One-dimensional L\'{e}vy process $(M=1)$.}\label{sec:21}
Now the nonlocal operator takes the form
\begin{align*}
	J^\alpha[\phi](t,x)= \int_{\mathbb{R}\setminus\{0\}}
	\big[\phi(t,x+\eta^\alpha(t,x,z))-\phi(t,x)-\eta^\alpha(t,x,z)
	D  \phi\big]k(z)dz,
\end{align*}
or, after two integrations by parts (more details are given below),
\begin{align}\label{J21}
	J^\alpha[\phi](t,x)=J^{\alpha,+}[\phi](t,x)
	+J^{\alpha,-}[\phi](t,x)-\tilde{b}^\alpha(t,x) D  \phi,
\end{align}
where $\tilde{b}^\alpha(t,x)=\int_{-\infty}^{\infty} \partial_z^2\eta^\alpha(t,x,z)\tilde{k}(z) dz$,
\begin{align}
	J^{\alpha,\pm}[\phi] & =\pm \int_0^{\pm\infty}
	\partial_z^2\big[\phi(t,x+\eta^\alpha(t,x,z))\big]\tilde{k}(z)dz,
	\label{J21pm}\\
	\tilde{k}(z) & = 
	\begin{cases} \int_{-\infty}^z\int_{-\infty}^wk(r)dr\, dw, 
		&\text{for}~z< 0\\
		\int_{z}^\infty\int_w^{\infty} k(r)dr\, dw, &\text{for $z> 0$}.
	\end{cases}
	\nonumber
\end{align}
By (A.4') ($M=1$, $\gamma<2$), $0\leq \tilde
k(z) \leq C
|z|^{1-\gamma}e^{-(\Lambda+\eps)|z|}$ and $\tilde k$ is integrable.

Note that $\tilde{b}^\alpha$ is bounded and $x$-Lipschitz, and that
$\tilde{b}^\alpha  D  \phi$ can be discretized using quadratures and
finite differences as in Appendix \ref{sec:Lh}. This leads to a first order monotone
(upwind) approximation -- we skip the standard details.

We propose the following approximation of $J^{\alpha,\pm}[\phi]$:
\begin{equation}\label{appr21}
	J^{\alpha,\pm}_h[\phi](t,x)=
	\sum_{n=0}^\infty\Delta_{zz,\Dz}
	\big[i_h\phi(t,x+\eta^\alpha(t,x,z_{n}))\big]
	\tilde{k}_{h,n}^\pm,
\end{equation}
where $z_n=n\Dx$ (not $n\Dz$!), $\Delta_{zz,\Dz}$ is defined in
\eqref{Delta}, the $x$-interpolation $i_h$ satisfies \eqref{eq:lint} and
\eqref{int_error_bnd}. Moreover,
\begin{equation*}
	\tilde{k}_{h,n}^+=
	\int_{z_n}^{z_{n+1}}\tilde{k}(z)dz
	\quad\text{and}\quad
	\tilde{k}_{h,n}^-=
	\int_{z_{-n-1}}^{z_{-n}}\tilde{k}(z)dz.
\end{equation*}
The approximation \eqref{appr21} can be derived from \eqref{J21pm} using
quadrature, finite differencing, and interpolation. 

To obtain \eqref{J21pm} and \eqref{J21}, we integrate by parts twice:
\begin{align*}
	&\int_0^\infty \big[\phi(t,x+\eta^\alpha(t,x,z))-\phi(t,x)
	-\eta^\alpha(t,x,z)D\phi\big]k(z)dz\\
	& = \Big[\big[\phi(t,x+\eta^\alpha(t,x,z))-\phi(t,x)
	-\eta^\alpha(t,x,z)D\phi\big]\big(-\int_z^\infty k(w)dw\big)\Big]_{z=0}^{z=\infty}\\
	&\qquad+\int_0^\infty \partial_z\big[\phi(t,x+\eta^\alpha(t,x,z))
	-\eta^\alpha(t,x,z)D\phi \big]\big(\int_z^\infty k(w)dw\big) dz\\
	& = 0+\Big[\partial_z\big[\phi(t,x+\eta^\alpha(t,x,z))
	-\eta^\alpha(t,x,z)D\phi\big]\big(-\tilde{k}\big)(z)\Big]_0^\infty\\
	&\quad+\int_0^\infty
	\partial_z^2\big[\phi(t,x+\eta^\alpha(t,x,z))
	-\eta^\alpha(t,x,z)D\phi\big]\tilde{k}(z)dz\\
	&= 0+0+\int_0^\infty
	\partial_z^2\big[\phi(t,x+\eta^\alpha(t,x,z))\big]\tilde{k}(z)dz
	-D\phi\int_0^\infty
	\partial_z^2\eta^\alpha(t,x,z)\tilde{k}(z) dz.
\end{align*}
In view of this result and similar computations for 
the integral on $(-\infty,0)$, \eqref{J21} follows. These
computations are rigorous if 
$\phi(t,x+\eta),\partial_z\phi(t,x+\eta),\partial_z^2\phi(t,x+\eta)$ and 
$\eta,\partial_z\eta,\partial_z^2\eta$ are $z$-integrable 
and bounded by $e^{\Lambda|z|}$ at infinity.

The approximation is consistent and has the error expansion
\begin{align*}
	J^{\alpha,\pm}[\phi](t,x)= J_h^{\alpha,\pm}[\phi](t,x) + E^\pm_Q
	+ E^\pm_{\mathrm{FDM}}+E^\pm_I,
\end{align*}
where $E_Q$, $E_{\mathrm{FDM}}$, $E_I$ have the same meaning as in Section \ref{sec:11}, and these 
terms can be estimated as follows:
\begin{align*}
	&|E_Q^\pm|\le \Delta x \,|\partial_z^3\phi(\cdot+\eta^\alp)|_0
	\int_{\R} \tilde k(z) dz,\\
	&|E_{\mathrm{FDM}}^\pm|\leq \frac1{24} \Dz^2\,
	|\partial_z^4\phi(\cdot+\eta^\alp)|_0\int_{\R} \tilde k(z) dz, \\ 
	&|E_I^\pm|\leq 4\frac{\Delta x^2}{\Dz^2}\, 
	|D_x^2\phi(\cdot+\eta^\alp)|_0 
	\int_{\R} \tilde k(z) dz.
\end{align*}

The proposed approximation is first order accurate if $\Dz=\Dx^{1/2}$, it 
is monotone satisfying \eqref{monJ}, and \eqref{barJ} holds if $\Dz=\Dx^{1/2}$. 
These properties follow as in Section \ref{sec:11}, since 
$J^{\alpha,\pm}_h[\phi](t,x_{\bar\beta})$ can be written as
\begin{align*}
	&\frac{1}{\Delta z^2}\tilde{k}^\pm_{h,0}
	\big[(i_h\phi)(t,x+\eta^\alpha(t,x_{\bar\beta},z_{\mp1}))
	-\phi(t,x_{\bar\beta})\big]\\
	&+\frac{1}{\Delta z^2}\sum_{n=1}^{\infty}
	(\tilde{k}^\pm_{h,n+1}-2\tilde{k}^\pm_{h,n}+\tilde{k}^\pm_{h,n-1})
	\big[(i_h\phi)(t,x+\eta^{\alpha}
	(t,x_{\bar\beta},z_{\pm n}))-\phi(t,x_{\bar\beta})\big],
\end{align*}
and, by convexity of $\tilde k(z)$ on $(0,\infty)$ and $(-\infty,0)$,
$$
\tilde{k}^\pm_{h,n+1}-2\tilde{k}^\pm_{h,n}+
\tilde{k}^\pm_{h,n-1}\geq0\qquad\text{for}\qquad n\geq 1.
$$
Moreover, $\bar j_{\beta}^{\alp,n}$ equals
$\frac2{\Dz^2}(\tilde k_{h,0}^+-\tilde
k_{h,1}^++\tilde k_{h,0}^--\tilde k_{h,1}^-)=\mathcal
O(\Dx^{2-\gamma}/\Dz^2)$ plus a $\mathcal O(\Dx^{-1})$ term from the
discretization of the $\tilde b^\alp$-term in \eqref{J21}. When
$\Dz=\Dx^{1/2}$ the leading term is the $\mathcal O(\Dx^{-1})$ term.

\subsubsection{Multi-dimensional L\'{e}vy process $(M>1)$.}
\label{sec:2M}
Writing $J^\alpha[\phi]$ in polar coordinates and performing two
integrations by parts in the radial direction leads to
\begin{align}\label{J-redef}
	J^\alpha[\phi](t,x) &
	={\tilde J}^{\alpha}[\phi](t,x)-\tilde{b}^\alpha(t,x) D\phi,
\end{align}
where $\tilde{b}^\alpha(t,x)=\int_{|y|=1}\int_0^\infty
\partial_r^2\big[\eta^\alpha(t,x,ry)\big]\tilde{k}(r,y)dr\, dS_y$ and
\begin{align*}
	{\tilde J}^{\alpha}[\phi](t,x) & =\int_{|y|=1}\int_0^\infty \partial_r^2
	\big[\phi(t,x+\eta^\alpha(t,x,ry))\big]
	\tilde{k}(r,y)dr\,dS_y,\\
	\tilde{k}(s) & =\int_s^\infty\int_w^\infty r^{M-1}k(ry)dr\,dw.
\end{align*}
By (A.4') ($\gamma<2$), $\tilde k(r,y) 
\leq C r^{1-\gamma}e^{-(\Lambda+\eps)r}$ 
and thus $\tilde k$ is $r$-integrable uniformly in $y$. 
Note that $\tilde{b}^\alpha$ is bounded and $x$-Lipschitz, and that
$\tilde{b}^\alpha D\phi$ can be discretized using quadrature and
finite differencing as in Appendix \ref{sec:Lh}. 
This leads to a first order  monotone (upwind) approximation -- we skip the standard details. 

We propose the following approximation of $\tilde J^{\alpha}[\phi]$:
\begin{align}\label{appr2M}
	&{\tilde J}^{\alpha}_h[\phi](t,x)=\int_{|y|=1} 
	\sum_{n=0}^\infty\Delta_{rr,\Dz}
	\big[(i_h\phi)(t,x+\eta^\alpha(t,x,r_{n}y))\big]
	\tilde{k}_{h,n}(y) \,dS_y,
\end{align}
where $r_n=n\Dx$ (not $n\Dz$!), $\Delta_{rr,\Dz}$ is defined in \eqref{Delta}, the 
$x$-interpolation $i_h$ satisfies \eqref{eq:lint} and 
\eqref{int_error_bnd}, and
\begin{gather*}
	\tilde{k}_{h,n}(y)
	=\int_{r_n}^{r_{n+1}}\tilde{k}(r,y)dz.
\end{gather*}
The approximation \eqref{appr2M} follows from \eqref{J-redef} by
quadrature, finite differencing, and interpolation, and the derivation of
\eqref{J-redef} is rigorous provided the functions 
$\phi(t,x+\eta),D_z\phi(t,x+\eta),D_z^2\phi(t,x+\eta)$ and
$\eta,D_z\eta,D_z^2\eta$ are $z$-integrable 
and bounded by $e^{\Lambda|z|}$ at infinity.

The approximation is consistent and has the error expansion
\begin{align*}
	{\tilde J}^{\alpha}[\phi](t,x)= {\tilde J}_h^{\alpha}[\phi](t,x)
	+ E_Q+ E_{\mathrm{FDM}}+E_I,
\end{align*}
where $E_Q$, $E_{\mathrm{FDM}}$, $E_I$ have the same meaning 
as in Section \ref{sec:21}, and can be estimated as follows:
\begin{align*}
	&|E_Q|\le \Delta x \,|D_z^3\phi(\cdot+\eta^\alp)|_0M_{\tilde k},\\
	&|E_{\mathrm{FDM}}|\leq \frac1{24} \Dz^2\,
	|D_z^4\phi(\cdot+\eta^\alp)|_0M_{\tilde k}, \\
	&|E_I|\leq 4\frac{\Delta x^2}{\Dz^2}\, 
	|D_x^2\phi(\cdot+\eta^\alp)|_0 M_{\tilde k},
\end{align*}
where $M_{\tilde k}:=\int_{|y|=1}\int_0^\infty \tilde k(r,y) dr\, dS_y$. 
Whenever $\Dz=\Dx^{1/2}$, this is a first order
approximation. Moreover, the approximation is monotone satisfying
\eqref{monJ} and, whenever $\Dz=\Dx^{1/2}$, it also satisfies
\eqref{barJ}. This follows as in Section \ref{sec:11} since 
$\tilde J^{\alpha}_h[\phi](t,x_{\bar\beta})$ can be written as an 
integral over $\{|y|=1\}$ with integrand
\begin{align*}
	&\frac{1}{\Delta z^2}
	\big[(i_h\phi)(t,x_{\bar\beta}+\eta^\alpha(t,x_{\bar\beta},r_{-1}y))
	-\phi(t,x_{\bar\beta})\big]\tilde{k}_{h,0}(y)\\
	&\qquad +\frac{1}{\Delta z^2}\sum_{n=1}^{\infty}
	\big(\tilde{k}_{h,n+1}(y)-2\tilde{k}_{h,n}(y)+\tilde{k}_{h,n-1}(y)\big)
	\\ & \qquad \qquad \qquad\quad \times 
	\Big[(i_h\phi)(t,x_{\bar\beta}
	+\eta^{\alpha}(t,x_{\bar\beta},r_ny))
	-\phi(t,x_{\bar\beta})\big]\Big].
\end{align*}
Furthermore, for each fixed $y$, $\tilde k(r,y)$ is convex 
on $(0,\infty)$ and thus
$$
\tilde{k}_{h,n+1}(y)-2\tilde{k}_{h,n}(y)+\tilde{k}_{h,n-1}(y)\geq 0
\qquad\text{for $n\geq 1$ and $|y|=1$.}
$$

\begin{rem}$\quad$
\smallskip

a. (Order of schemes)
In general, our discretizations of the non-local term 
in \eqref{eq:integro_pde} are at most first order accurate. In the case
$\gamma\in[1,2)$, a first order rate is obtained by choosing $\Dz=\Dx^{1/2}$.
Higher order discretizations can be derived using higher order quadrature and
interpolation rules, but the resulting discretizations are
not monotone in general. On the other hand, if
$\eta\equiv z$, then interpolation is not
needed and consequently the monotone discretizations 
of Section \ref{g12} are 2nd order accurate.
\smallskip

b. (Remaining discretizations)
To obtain fully discrete schemes it remains to discretize the various terms
involving $\tilde b^\alp D\phi$, for example by quadrature and
finite differencing, cf.~Appendix \ref{sec:Lh}. 
In applications, the densities $\hat k$ and $\tilde k$ can often be explicitly
calculated, e.g., using incomplete gamma functions as in
\cite{AO07}. Otherwise these quantities also have to be computed by
quadrature. Furthermore, regarding Sections \ref{sec:1M} and \ref{sec:2M}, it
also remains to discretize the surface integral in $y$. This discretization
does not pose any problems, neither numerically nor in the analysis, as long as positive
quadratures are used. The details are left to the reader.
\smallskip

c. (Increasing efficiency)
From a practical point of view in terms computational 
efficiency, quadratures should be implemented using FFT. 
This is standard and we refer to, e.g., \cite{DFV05} for the details.
\smallskip

d. (Generalization I)
The above approximations (with obvious modifications) 
also apply to integral terms of the type
$$
J^\alpha[\phi](t,x)
=\sum_{i=1}^M \int_{\mathbb{R}\setminus\{0\}}
\big(\phi(t,x+\eta_{i}^\alpha(t,x,z))
-\phi(t,x)-\eta_{i}^\alpha(t,x,z)D\phi(t,x)\big)k_i(z)dz.
$$
Such terms appear in $M$-dimensional L\'{e}vy models based on $M$ 
independent Poisson random measures coming from one-dimensional L\'{e}vy processes. 
This is a rich class of models with many applications. 
For more information and analysis of such
models we refer to the book \cite{Oksendal:2005}.
\smallskip

e. (Generalization II)
With obvious modifications, our approximations also apply to
linear and non-linear equations involving the fractional Laplace
operator
$$
(-\Delta)^\alp u(x)=c_\alp\int_{|z|>0}
\frac{u(x+z)-u(x)-zDu(x)}{|z|^{N+2\alp}}dz, 
\qquad \alp\in(0,1),
$$
where $x,z\in\R^N$ and $c_\alp$ is a constant, in which 
case the L\'{e}vy measure takes the form $\nu(dz)=|z|^{-N-2\alp}dz$. 
This measure satisfies \ref{A3} except for the ``exponential 
decay at infinity''  requirement. It is straightforward to recast 
the entire theory to allow for a fractional Laplace setting 
where assumption (A.3) is replaced by
$$
\int_{|z|>0}|z|^2\wedge 1\,\nu(dz)<\infty.
$$
\end{rem}

\section{Error estimates for a switching system approximation}
\label{sec:sw}

In this section we obtain error estimates for a switching 
system approximation of \eqref{eq:integro_pde}--\eqref{eq:bdcond}.
This result, which has independent interest, plays a crucial role in the 
proof of Theorem \ref{err_est} in Section \ref{sec:PfErr}.

The switching system will be written as
\begin{align}\label{swi_system} 
	F_i(t,x,v,\partial_t v^i, Dv^i, D^2v^i,u^i(t, \cdot)) &= 0 
	\quad \text{in}\quad 
	Q_T, \quad i\in \{1,2,....., m\},\\
	\label{eq:BC2}
	v(0,x)&= (g(x), \dots, g(x)) 
	\quad \textrm{in}\quad\mathbb{R}^N,
\end{align}
where $v=(v_1,\dots,v_m)$ is in $\mathbb{R}^m$ and for sets $\mathcal{A}_i$
such that $\cup_i\mathcal{A}_i=\mathcal{A}$,
\begin{align*}
	&F_i(t,x,r,p_t, p_x, X, \phi(\cdot))\\
	&\quad =\max\big\{p_t+\sup_{\alpha\in\mathcal{A}_i}
	\big[\mathcal{L}^{\alpha}(t,x, r_i, p_x,X)
	-J^{\alpha}[\phi](t,x)\big];r_i-\mathcal{M}^i r\big\},\\
	&\mathcal{L}^{\alpha}(t,x,r,p,X):=-\textrm{tr}\big(a^\alpha(t,x)X\big)
	-b^\alpha(t,x)\cdot p+c^\alpha(t,x)r-f^\alpha(t,x),\\
	& \mathcal{M}^i r=\min_{j\neq i}\{r_j+k\}, \quad k>0,
\end{align*}
for $\big(t,x,r,p_t,p_x, X\big)\in\mathbb{R}\times
\mathbb{R}^N\times\mathbb{R}^m\times\mathbb{R}\times\mathbb{R}^N\times\mathbb{S}^N$ 
and any smooth bounded function $\phi$. 
The operator $J^\alpha[\phi]$ is defined below \eqref{eq:integro_pde}.

In the pure PDE case, such approximations have been 
studied in, e.g., \cite{Capuzzo:1984,Evans:1979,Barles:2006jf}. Here we
extend the error estimates of \cite{Barles:2006jf} to non-local Bellman
equations. In a complimentary article \cite{Biswas:2007td}, we
develop a viscosity solution 
theory covering switching systems like \eqref{swi_system}. We refer
to that paper for the precise definition of viscosity solutions and
proofs of the associated results utilized herein. 
If assumptions \ref{A1} -- \ref{A4} of Section \ref{sec:FD} hold, then we
have the following well posedness result \cite{Biswas:2007td}:

\begin{prop}\label{wellpos_2}
Assume that conditions \ref{A1}--\ref{A4} hold. 
There exists a unique viscosity solution $v$
of \eqref{swi_system}--\eqref{eq:BC2}, satisfying $|v|_1\le C$ 
for a constant $C$ depending only on $T$ and $K$ from \ref{A1}--\ref{A3}. 
Furthermore, if $w_1,w_2$ are respectively viscosity sub and 
supersolutions of \eqref{swi_system} satisfying 
$w_1(0,\cdot)\le w_2(0,\cdot)$, then $w_1\le w_2$.
\end{prop}

Before we continue, we need the following remark.

\begin{rem}\label{ext_rmk}
The functions $\sigma^\alpha, b^\alpha, c^\alpha, f^\alpha, \eta^\alpha$ 
are only defined for times $t\in[0,T]$. 
But they can be easily extended to times 
$[-r, T + r]$ for any $r>0$ in such a way that
\ref{A1} -- \ref{A3} still hold. In view of Proposition \ref{wellpos_2}
we can then solve the initial value problem up to time $T + r$ or,
by using a translation in time, we may start from time $-r$. We will
use these facts several times below.
\end{rem}

By equi-continuity and the Arzela-Ascoli theorem it easily follows that
the each component of the solution of
\eqref{swi_system}--\eqref{eq:BC2} converges locally
uniformly to the solution of
\eqref{eq:integro_pde}--\eqref{eq:bdcond} as $k\rightarrow 0$.
To derive an error estimate we use Krylov's method of shaking the coefficients coupled with an
idea of P.-L. Lions as in \cite{Barles:2006jf}. 
We need the following auxiliary system
\begin{align}\label{swi_system_auxi} 
	F_i^\epsilon(t,x,v^\epsilon,\partial_t v_i^\epsilon, 
	Dv_i^\epsilon, D^2v_i^\epsilon,v_i^\epsilon(t, \cdot)) &= 0 \quad \text{in}
	\quad Q_{T+\epsilon^2},\quad i\in \{1,\dots, m\},\\ \nonumber
	v^\epsilon(0,x)&=(g(x),\dots,g(x)) 
	\quad \textrm{in}\quad\mathbb{R}^N,
\end{align}
where $v^\epsilon = (v_1^\epsilon, \dots, v_m^\epsilon)$ and
\begin{align*}
	F_i^\epsilon(t,x,r,p_t, p_x, X, \phi(\cdot))
	=\max\big\{&p_t+\sup_{\alpha\in\mathcal{A}_i;|e|\le \epsilon;0\le s\le \epsilon^2}
	\big(\mathcal{L}^{\alpha}(t+s,x+e, r_i, p_x, X)\\&
	\quad -J^{\alpha}(t+s,x+e)\phi\big); r_i-\mathcal{M}^i r\big\}.
\end{align*}
The operators $\mathcal{L}$, $J$, and $\mathcal{M}$ are as previously defined. 

Note that we have used the extension of the data
mentioned in Remark \ref{ext_rmk}. By regularity and
continuous dependence results from \cite{Biswas:2007td} we have
\begin{prop}\label{aux_wellpos}
Assume that \ref{A1}--\ref{A4} hold. There exists a unique
viscosity solution $v^\epsilon: \bar{Q}_{T+\epsilon^2}\rightarrow\mathbb{R}$ 
of \eqref{swi_system_auxi} satisfying
\begin{align*}
	|v^\epsilon|_1+\frac{1}{\epsilon}|v^\epsilon-v|_0 
	\le C,
\end{align*} 
where $C$ depends $T$ and $K$. Furthermore, if $w_1$
and $w_2$ are respectively sub and supersolutions of \eqref{swi_system_auxi}
satisfying $w_1(0,\cdot)\le w_2(0,\cdot)$, then $w_1\le w_2$.
\end{prop}

We are now in a position to prove the following main result of this section:

\begin{thm}\label{thm:switch_rate}
Assume that \ref{A1}--\ref{A4} hold. 
If $u$ and $v$ are respectively viscosity solutions of 
\eqref{eq:integro_pde}--\eqref{eq:bdcond} and 
\eqref{swi_system}--\eqref{eq:BC2}, then for sufficiently small $k$,
\begin{align*}
	0\le v_i-u\le Ck^{\frac{1}{3}}, 
	\qquad i\in\{1,\dots,m\},
\end{align*}
where $C$ depends only on $K$ and $T$.
\end{thm}

\begin{proof}
Since $w=(u,\dots,u)$ is a viscosity subsolution of \eqref{swi_system}, the
first inequality $u\le v_i$ follows from the comparison principle.

The second inequality will be obtained in the following. Since 
$v^\epsilon$ is the viscosity solution of
\eqref{swi_system_auxi}, it follows that
\begin{align*}
	\partial_t v_i^\epsilon +\sup_{\alpha\in\mathcal{A}_i}
	\big(\mathcal{L}^\alpha(t+s,x+e,v_i^\epsilon(t,x),Dv_i^\epsilon, D^2 v_i^\epsilon)
	-J^\alpha(t+s,x+e) v_i^\epsilon\big)\le 0
\end{align*}
in $Q_{T+\epsilon^2}$ in the viscosity sense, $i=1,\dots,m$. 
After a change of variable, we conclude that
for every $0\le s\le \epsilon^2$ and $|e|\le \epsilon $,
$v^\epsilon(t-s, x-e)$ is a viscosity subsolution of the uncoupled system
\begin{align}\label{uncoupled}
	\partial_t w_i^\epsilon
	+\sup_{\alpha\in\mathcal{A}_i}\big(\mathcal{L}^\alpha(t, x,w_i^\epsilon, 
	Dw_i^\epsilon, D^2 w_i^\epsilon)
	-J^\alpha(t, x) w_i^\epsilon\big)
	=0\quad\textrm{in}\quad Q^\epsilon_{T},
\end{align}
where $Q^\epsilon_{T}:=(\epsilon^2,T)\times\R^N$.
Now set $v_\epsilon := v^\epsilon\star\rho_\epsilon$, 
where $\rho_\epsilon$ is the mollifier defined in \eqref{eq:mollifier}. 
A Riemann-sum approximation shows that this function is the 
limit of convex combinations of viscosity subsolutions 
$v(t-s,x-e)$ of the convex system \eqref{uncoupled}. Hence
$v_\epsilon$ is also a viscosity subsolution of \eqref{uncoupled} (see
the appendix of \cite{Jakobsen:2005tp} for more details). 
On the other hand, since $v^\epsilon $ is a 
continuous subsolution of \eqref{swi_system_auxi},
\begin{align*}
	v_i^\epsilon \le \min_{j\neq i} v_j^\epsilon 
	+k \quad \textrm{in}~Q_{T+\epsilon^2},~i\in\{1,\dots,m\}.
\end{align*} 
It follows that $\max_i v_i^\epsilon(t,x) -\min v_j^\epsilon(t,x)\le k$ 
in $Q_{T+\epsilon^2}$, and therefore
\begin{align*}
	|v_i^\epsilon-v_j^\epsilon|_0\le k,
	\qquad i,j\in\{1,\dots,m\}.
\end{align*}
Then, by the definition and properties of $v_\epsilon$,
\begin{align*}
	|\partial_t v_{\epsilon i}-\partial_tv_{\epsilon j}|_0
	\le C \frac{k}{\epsilon^2}
	\qquad\text{and}\qquad
	|D^nv_{\epsilon i}-D^n v_{\epsilon j}|_0
	\le C\frac{k}{\epsilon^n},
\end{align*}
for $n\in\mathbb{N},~i,j\in\{1,\dots,m\}$, where $C$ depends only on
$\rho$, $T$, and $K$. For $\epsilon<1$, it follows that
\begin{align*}
	&|\partial_t v_{\epsilon j}
	+\sup_{\alpha\in\mathcal{A}_i}
	\big(\mathcal{L}^\alpha(t, x,v_{\epsilon j}(t,x), Dv_{\epsilon j}, D^2 v_{\epsilon j})
	-J^\alpha(t, x) v_{\epsilon j}\big)\\
	&\quad -\partial_t v_{\epsilon i}
	-\sup_{\alpha\in\mathcal{A}_i}
	\big(\mathcal{L}^\alpha(t, x,v_{\epsilon i}(t,x), Dv_{\epsilon i}, D^2 v_{\epsilon i})
	-J^\alpha(t, x) v_{\epsilon i}\big)|\le C
\frac{k}{\epsilon^2}
\end{align*}
and, since $v_\epsilon$ is subsolution of \eqref{uncoupled},
\begin{align*}
	\partial_t v_{\epsilon i}
	+\sup_{\alpha\in{\mathcal{A}}}\big(\mathcal{L}^\alpha
	(t,x, v_{\epsilon i}(t,x), Dv_{\epsilon i}, D^2 v_{\epsilon i})
	-J^\alpha(t, x) v_{\epsilon i}\big)
	\le C\frac{k}{\epsilon^2}\quad \textrm{in}~Q_T^\epsilon,
\end{align*}
where the constant $C$ depends on $\rho$, $T$, and $K$. From this
inequality it is easy to see that $v_{\epsilon i}-te^{Kt}C\frac{k}{\epsilon^2}$ is a 
subsolution of \eqref{eq:integro_pde} restricted to $Q_T^\epsilon$. Hence, by the comparison
principle,
\begin{align*}
	v_{\epsilon i}-u\le e^{Kt}
	\Big(|v_{\epsilon i}(\epsilon^2,\cdot)-u(\epsilon^2,\cdot)|_0
	+Ct\frac{k}{\epsilon^2}\Big)\quad
	\text{in $Q_T^\epsilon$, $i\in \{1,\dots,m\}$.}
\end{align*}
By regularity in time, $|u(t,\cdot)-v_i(t,\cdot)|_0 \le
(|u|_1+|v_i|_1)\epsilon$, and by Proposition \ref{aux_wellpos} and
properties of mollification we conclude that
\begin{align*}
	v_i-u\le v_i-v_{\epsilon i} +v_{\epsilon i}-u\le 
	C(\epsilon+\frac{k}{\epsilon^2}) 
	\quad\text{in $Q_T$, $i\in\{1,\dots,m\}$.}
\end{align*}
Now the theorem follows by minimizing with respect to $\epsilon$.
\end{proof}

\section{The Proof of Theorem \ref{err_est}}

\label{sec:PfErr}
To prove Theorem \ref{err_est} we will use different arguments for the 
upper and lower bounds. The upper bound, part (a), is the ``easy'' part, and it 
is essentially a reformulation of the general
upper bound established in \cite{Jakobsen:2005tp}. We skip the
details, and prove only part (b) which is a new result.

\begin{proof}[Proof of Theorem \ref{err_est} (b)]
Without loss of generality we will assume that $\mathcal{A}$ is
finite:
$$
\mathcal{A} = \{\alpha_1,\alpha_2,\dots,\alpha_m\}.
$$
The proof of this statement is similar to the one given in
\cite{Barles:2006jf} in the pure PDE case and relies on assumption \ref{A1}. 
Now we follow \cite{Barles:2006jf} and use a switching system
approximation to construct approximate
supersolutions of \eqref{eq:integro_pde} which are point-wise
minima of smooth functions and approximates the viscosity solution of
\eqref{eq:integro_pde}--\eqref{eq:bdcond}. Consider
\begin{align}\label{aux_switch} 
	F_i^\epsilon\big(t,x, v^\epsilon,\partial_t v_i^\epsilon,D v_i^\epsilon,
	D^2v_i^\epsilon, v_i^\epsilon(t,\cdot) \big) & = 0 
	\quad \textrm{in}~Q_{T+2\epsilon^2}\\
	\nonumber v^\epsilon(0,x) &= v_0(x)
	\quad\text{in}\quad \mathbb{R}^N,
\end{align} 
where $v^\epsilon = (v_1^\epsilon,\dots,v_m^\epsilon)$, $v_0=(g,\dots, g) $, and
\begin{align*}
	&F_i^\epsilon\big(t,x, r, p_t,p_x,X,\phi(t,\cdot) \big)\\
	&=\max\big\{ p_t+\min_{0\le s\le\epsilon^2, |e|
	\le \epsilon}\big(\mathcal{L}^{\alpha_i}
	\big(t+s-\epsilon^2,x+e, r, p_t,p_x, X\big)\\
	&\qquad\qquad\quad
	-J^{\alpha_i}(t+s-\epsilon^2, x+e)\phi\big); 
	r_i-\mathcal{M}^i r\big\},
\end{align*} 
where $\mathcal{L},J$ and $\mathcal{M}$ are
defined below \eqref{swi_system} in Section \ref{sec:sw}. 
This new problem is well-posed and each component 
of the solution of this switching system will converge to the viscosity 
solution of \eqref{eq:integro_pde} as $k,\epsilon\to 0$:

\begin{lem}\label{aux_switch_well}
Assume that conditions \ref{A1}--\ref{A4} hold. 
There exists a unique solution 
$v^\epsilon$ of \eqref{aux_switch} satisfying
\begin{align*}
	|v^\epsilon|_1\le \bar{K},\quad 
	\max_{i,j}|v_i^\epsilon-v_j^\epsilon|_0\le k,
	\quad \text{and for $k$ small,
	$\max_{i}|u-v_i^\epsilon|\le 
	C(\epsilon+ k^{\frac 13})$,}
\end{align*}
where $\bar K,C$ only depend on $T$ and $K$ from \ref{A2}--\ref{A4}.
\end{lem}

\begin{proof}
From \cite{Biswas:2007td} we have the existence and uniqueness of a viscosity solution, 
and moreover the uniform bounds
\begin{align*}
	|v^\epsilon|_1\le \bar{K}\quad\text{and}\quad
	|v^\epsilon-v^0|_0\le C\epsilon,
\end{align*}
where $v^0$ is the unique viscosity solution of \eqref{aux_switch}
corresponding to $\epsilon = 0$. The last inequality in the lemma now
follows since $|u-v_i^0|_0\le C k^\frac 13$ by Theorem \ref{thm:switch_rate}.
To second inequality follows since arguing as in the proof of Theorem
\ref{thm:switch_rate} leads to $0\le \max_i v_i^\epsilon-\min_j v_j^\epsilon 
\le k\quad \text{in} \quad Q_{T+2\epsilon^2}$.
\end{proof}

Next we time-shift and mollify $v^\epsilon$. For $i=1,\dots,m$, set
\begin{align*}
	\bar{v}_i^\epsilon(t,x):= v_i^\epsilon(t+\epsilon^2,x),
	\qquad 
	v_{\epsilon i}(t,x):= \rho_\epsilon \star \bar{v}_i^\epsilon(t,x),
\end{align*}
where $\rho$ is defined in \eqref{eq:mollifier}. 
Note that $\mathrm{supp} (\rho_\eps)\subset (0,\eps^2)\times B(0,\eps)$ 
and that the functions
$v^\epsilon$, $\bar{v}^\epsilon$, $v_\epsilon$ are
well-defined respectively on $Q_{T+2\eps^2}$, 
$(-\epsilon^2, T+\epsilon^2]\times \mathbb{R}^N $, $Q_{T+\epsilon^2}$. 
By Lemma \ref{aux_switch_well} and properties of mollifiers,
\begin{align}
	\nonumber
	&|\bar{v}^\epsilon|_1\le \bar{K},\quad
	|\bar{v}^\epsilon-v^\epsilon|_0\le
	\bar{K}\epsilon,\\
	\label{eq:aux_well_2}
	&\max_{i,j}|v_{\epsilon i}-v_{\epsilon j}| 
	\le C(k+\epsilon)\quad\text{in}\quad Q_{T+\epsilon^2},\\
	&\nonumber \max_{i}|u-v_{\epsilon, i}|\le 
	C(\epsilon + k^{\frac 13})\quad \text{in}\quad Q_T,
\end{align} 
where $C$ depends only on $\rho$ and $K,T$ from \ref{A2}--\ref{A4}. 
A supersolution of \eqref{eq:integro_pde} can now be produced by setting
$$
w:= \min_{i} v_{\epsilon i}.
$$

\begin{lem}\label{ap_smooth_sup} 
Assume that conditions \ref{A1}---\ref{A4} hold and 
$\epsilon \le (8 \sup_i[v_i^\epsilon]_1)^{-1} k$. For every
$(t,x)\in Q_T$, if $j := {\argmin}_{i}v_{\epsilon i}(t,x)$,
\begin{align}\label{eq:aux_well_smooth}
	\partial_t v_{\epsilon j}+\mathcal{L}^{\alpha_j}
	\big(t,x,v_{\epsilon j}(t,x), D v_{\epsilon j}(t,x),D^2 v_{\epsilon j}(t,x)\big)
	-J^{\alpha_j}(t,x) v_{\epsilon j} \ge 0.
\end{align}
\end{lem}
We postpone the proof of this lemma. From this lemma it follows that
$w$ is an approximate supersolution to the scheme \eqref{abs_scheme} when
$\epsilon \le (8 \sup_i[v_i^\epsilon]_1)^{-1} k$:
\begin{equation}\label{appr_sch}
	S\big(h,t,x,w(t,x),[w]_{t,x}\big)
	\ge -E_2(\bar{K},h,\epsilon)
	\quad \text{in}\quad \mathcal{G}_h^+,
\end{equation}
where $\bar{K}$ comes from Lemma \ref{aux_switch_well}. 
To see this, let $(t,x)\in Q_T$ and set $j := {\argmin}_{i}v_{\epsilon i}(t,x)$. 
At $(t,x)$, $w(t,x)= v_{\epsilon j}(t,x)$ 
and $w\le v_{\epsilon j}$ in $\mathcal{G}_h$. 
Hence (S1) implies that
$$
S\big(h, t, x, w(t,x), [w]_{t,x}\big)
\ge S\big(h,t,x,v_{\epsilon j}(t,x),[v_{\epsilon j}]_{t,x}\big).
$$
By consistency (S3)(ii) we have
\begin{align*}
	& S\big (h, t, x, v_{\epsilon j}(t,x),[v_{\epsilon j}]_{t,x}\big)\\
	& \ge \partial_t v_{\epsilon j}+ F\big(t,x,v_{\epsilon j}(t,x),D v_{\epsilon j},D^2 v_{\epsilon j},
	v_{\epsilon j}(t, \cdot)\big)- E_2(\bar{K}, h,\epsilon),\\
	&\ge \partial_t v_{\epsilon j}+\mathcal{L}^{\alpha_j}
	\big(t,x, v_{\epsilon j}(t,x), D v_{\epsilon j}(t,x), D^2 v_{\epsilon j}(t,x)\big)-J^{\alpha_j}(t,x)
	v_{\epsilon j}- E_2(\bar{K}, h,\epsilon),
\end{align*}
and \eqref{appr_sch} then follows from Lemma \ref{ap_smooth_sup}.

To derive the lower bound on the error $u_h- u$, we take $\epsilon =
(8 \sup_i[v_i^\epsilon]_1)^{-1} k$ and use \eqref{appr_sch} and
comparison Lemma \ref{cmpare_scm} to get
\begin{align*}
	u_h-w \le e^{\mu t}|(g_{h}-w(0,\cdot))^+|
	+2t e^{\mu t}E_2(\bar{K},h,\epsilon)
	\quad \text{in $\mathcal{G}_h$.}
\end{align*}
By \eqref{eq:aux_well_2}, $|w-u|\le C(\epsilon + k+ k^{\frac 13})$,
and hence
\begin{align*}
	u_h-u \le e^{\mu t}|(g_{h}-w(0,\cdot))^+|
	+2t e^{\mu t}E_2(\bar{K}, h,\epsilon)
	+C(\epsilon + k+ k^{\frac 13})\quad \text{in}
	\quad \mathcal{G}_h,
\end{align*}
possibly with a new constant $C$. Since $\epsilon=C k$, the
proof is complete by minimizing the right hand side with respect to $\epsilon$.
\end{proof}

\begin{proof}[Proof of Lemma \ref{ap_smooth_sup}]
We begin by fixing $(t,x)\in Q_T$ and set 
$j =\argmin_{i} v_{\epsilon i}(t,x) $. Then
\begin{align*}
	v_{\epsilon j}(t,x)-\mathcal{M}^j v_\epsilon(t,x) 
	=\max_{i\neq j}\big\{ v_{\epsilon j}(t,x)-v_{\epsilon i}-k\big\}\le -k.
\end{align*} 
Therefore, by the H\"{o}lder continuity of $\bar{v}^\epsilon$ and basic 
properties of mollifiers,
\begin{align*}
	&\bar{v}_j^\epsilon(t,x)-\mathcal{M}^j\bar{v}^\epsilon(t,x)
	\le -k +2\max_i[v_i^\epsilon]_1 2\epsilon
\end{align*}
and
\begin{align*}
	\bar{v}_j^\epsilon(s,y)-\mathcal{M}^j\bar{v}^\epsilon(s,y)
	\le -k+2\max_i[v_i^\epsilon]_1
	( 2\epsilon+|x-y|+|t-s|^{\frac 12}),
\end{align*}
for all $(s,y)\in Q_T$. Consequently, if $|x-y|< \epsilon$, $|t-s|< \epsilon^2$, and 
$\epsilon \le (8 \sup_i[v_i^\epsilon]_1)^{-1} k$, then 
\begin{align}\label{eq:obstacle_regular}
	\bar{v}_j^\epsilon(s,y)-\mathcal{M}^j\bar{v}^\epsilon (s,y)< 0.
\end{align}

To continue we need the following remark. Let $u^1,\dots, u^k$ be
functions satisfying \eqref{eq:obstacle_regular} at $(t,x)$, then any linear
combination $u^\lambda=\sum_{i=1}^k\lambda_iu^i$ with $\lambda_i\geq0$ and 
$\sum_{i=1}^k\lambda_i=1$, also satisfies \eqref{eq:obstacle_regular}
at $(t,x)$. If, in addition, $u^1,\dots, u^k$ are supersolutions of
\begin{align}\label{newsystem}
	\max\big\{ \partial_t u_{j}+\mathcal{L}^{\alpha_j}(t,x,u_j,Du_j, D^2u_j)
	-J^{\alpha_j}[u_j](t,x); u_j-\mathcal{M}^j u\big\}= 0,
\end{align}
then in view of \eqref{eq:obstacle_regular} they are also
supersolutions of the linear equation
\begin{align}\label{eq:convex-sup-solution.}
	\partial_tu_{j}+\mathcal{L}^{\alpha_j}(t,x,u_j,Du_j,D^2u_j)
	-J^{\alpha_j}[u_j](t,x)=0
\end{align}
at $(t,x)$. An easy adaptation of the proof of Lemma 6.3 in
\cite{Jakobsen:2005tp} then shows that $u^\lambda$ is also a
viscosity supersolution of \eqref{eq:convex-sup-solution.} at $(t,x)$.

A change of variables reveals that $\big\{\bar{v}^\epsilon(\cdot-s,
\cdot-e)\big\}_{s, e}$, $0\le s<\epsilon^2$, $|e|<
\epsilon$, is a family of supersolutions of
\eqref{newsystem} with $[\bar{v}_j^\epsilon-\mathcal{M}^j
\bar{v}^\epsilon](t-s,x-e )< 0.$ We note that by approximating
the function $v_{\epsilon j}$ by a Riemann sum, we see that it
is the limit of convex combinations of $\bar{v}^\epsilon(\cdot-s,\cdot-e)$. 
In view of the above remark, these convex
combinations are supersolutions of \eqref{newsystem}, and hence by
the stability result for viscosity supersolutions, so is the limit
$v_{\epsilon j}$. Finally, since this function is smooth it is also a
classical supersolution of \eqref{newsystem} 
and hence \eqref{eq:aux_well_smooth} holds.
\end{proof}

\appendix

\section{An example of a monotone discretization of $L^{\alpha}$}\label{sec:Lh}

Let $\{e_i\}_{i=1}^N$ be the standard basis of $\mathbb{R}^N$ and
$a_{ij}$ the $ij$-th element of the matrix $a$. 
Kushner and Dupuis \cite{Kushner:1992mq} suggest the
following discretization of $L^{\alpha}$ in \eqref{eq:integro_pde}:
\begin{align*}
	L_h^\alpha \phi :=
	\sum_{i=1}^N\Big[a_{ii}^\alpha\Delta_{ii}+\sum_{i\neq j}
	\big(a_{ij}^{\alpha+}\Delta_{ij}^+-a_{ij}^{\alpha-}\Delta_{ij}^-\big)
	+ b_i^{\alpha +}\delta_i^+-b_i^{\alpha-}\delta_i^{ -}\Big]\phi,
\end{align*} 
where $b^+=\max\{b,0\}$, $b^-=(-b)^+$, and
\begin{align*}
	\delta_i^{\pm} \phi(x) &= \pm\frac{1}{\Delta x}
	\big\{\phi(x\pm e_i\Delta x)-\phi(x)\big\},\\ \Delta_{ii} \phi(x) 
	& = \frac{1}{\Delta x^2}\big\{\phi(x+ e_i \Delta x)
	-2\phi(x)+\phi(x- e_i \Delta x)\big\},\\ 
	\Delta_{ij}^+ \phi(x) &= \frac{1}{2\Delta x^2}
	\big\{2\phi(x)+\phi(x+e_i\Delta x+e_j\Delta x)
	+\phi(x-e_i\Delta x-e_j\Delta x)\big\}\\ &\quad
	-\frac{1}{2\Delta x^2}\big\{\phi(x+e_i\Delta x)+\phi(x-e_i\Delta x)
	+\phi(x+e_j\Delta x)+\phi(x-e_j\Delta x)\big\}\\ \Delta_{ij}^-\phi(x) 
	& = \frac{1}{2\Delta x^2}\big\{2\phi(x)+\phi(x+e_i\Delta x-e_j\Delta x)
	+\phi(x-e_i\Delta x+e_j\Delta x)\big\}\\ &\quad 
	-\frac{1}{2\Delta x^2}\big\{\phi(x+e_i\Delta x)
	+\phi(x-e_i\Delta x)+\phi(x+e_j\Delta x)
	+\phi(x-e_j\Delta x)\big\}.
\end{align*}
By Taylor expansion it is easy to check that the truncation error is
given by \eqref{consistL}. Moreover $L_h$ can be written in the form \eqref{monL} with
\begin{gather*}
	l_{h,\beta,\beta\pm e_i}^{\alpha,n} = 
	\frac{1}{\Delta x^2}\Big[a_{ii}^\alpha(t,x)
	-\frac12\sum_{j\neq i} |a^\alpha_{ij}(t,x)|\Big]
	+\frac{b_i^{\alpha\pm}(t,x)}{\Delta x},\\ 
	l_{h,\beta,\beta+ e_i\pm e_j}^{\alpha,n}
	=\frac{a_{ii}^{\alpha \pm}(t,x)}{\Delta x^2},
	\qquad l_{h,\beta,\beta-e_i\pm e_j}^{\alpha,n}
	= \frac{a_{ii}^{\alpha \pm}(t,x)}{\Delta x^2},\qquad i\neq j,
\end{gather*}
and $l_{h,\beta,\bar\beta}^{\alpha,n}=0$ otherwise. This approximation
is monotone if $l_{h,\beta,\bar\beta}^{\alpha,n}\geq0$ for all
$\alp,\beta,\bar\beta,n$ and $\Dx>0$, which happens to 
be the case, e.g., if $a$ is diagonally dominant:
$$
a_{ii}^\alpha(t,x)-\sum_{j\neq i} 
|a_{ij}^\alpha(t,x)|\ge 0\quad 
\text{in $Q_T$, for each $\alpha\in\mathcal{A}$.}
$$

\end{document}